\documentclass[10pt,oneside,reqno]{amsart}

\usepackage[latin1]{inputenc}
\usepackage[T1]{fontenc}

\usepackage{lipsum} 

\usepackage[english,francais]{babel}

\usepackage{lmodern}
\rmfamily
\DeclareFontShape{T1}{lmr}{b}{sc}{<->ssub*cmr/bx/sc}{}
\DeclareFontShape{T1}{lmr}{bx}{sc}{<->ssub*cmr/bx/sc}{}

\usepackage{amssymb,amsmath,amsthm,amscd}
\usepackage{mathrsfs}
\usepackage{mathtools} 
\usepackage{thmtools,thm-restate}
\usepackage{dsfont} 
\usepackage{bbm}
\usepackage[mathscr]{euscript}

\usepackage{subfig}

\usepackage{tabularx} 
\usepackage{calc} 
\usepackage{graphicx} 

\usepackage{hyperref}
\hypersetup{
colorlinks=false,       
linkcolor=red,          
citecolor=green,       
filecolor=magenta,   
urlcolor=cyan     }

\usepackage{color}
\usepackage[normalem]{ulem}

\usepackage{appendix}

\usepackage{enumerate}
\usepackage[shortlabels]{enumitem} 
\setitemize{label=$\bullet$}

\usepackage{fancyhdr}
\makeatletter
\let\ps@plain=\ps@empty
\makeatother

\let\origdoublepage\cleardoublepage
\newcommand{\clearemptydoublepage}{%
  \clearpage
  {\pagestyle{empty}\origdoublepage}%
}
\let\cleardoublepage\clearemptydoublepage

\FrenchFootnotes 
\AddThinSpaceBeforeFootnotes 

\theoremstyle{plain} 
\newtheorem{theo}{Theorem}

\newtheorem{cor}[theo]{Corollary}

\renewcommand{\thmcontinues}[1]{bis} 

\newtheorem{lemma}[theo]{Lemma}
\newtheorem{prop}[theo]{Proposition}

\newtheorem*{theo*}{Theorem}
\newtheorem*{conj*}{Conjecture} 

\newcounter{claimcounter}

\theoremstyle{definition} 

\theoremstyle{remark} 
\newtheorem{remark}[theo]{Remark} 
\newtheorem*{remark*}{Remark}

\newtheoremstyle{note}
{3pt}
{3pt}
{}
{}
{\itshape}
{:}
{.5em}
{}

\newcommand{\N}{\mathbb{N}}

\newcommand{\R}{\mathbb{R}}

\newcommand{\D}{\mathcal{D}}

\newcommand{\A}{\mathcal{A}}

\newcommand{\ctree}{\mathcal{T}} 

\newcommand{\Co}{\mathcal{C}} 

\newcommand{\e}{\mathrm{e}}
\newcommand{\der}{\mathrm{d}}
\newcommand{\1}{\mathbbm{1}}
\newcommand{\argmax}{\mathrm{argmax}\hspace*{0.07cm}}
\newcommand{\argmin}{\mathrm{argmin}\hspace*{0.07cm}}

\newcommand{\arginf}{\mathrm{arginf}\hspace*{0.07cm}}

\title[Branching processes seen from their extinction time]{Branching processes seen from their extinction time via path decompositions of reflected Lévy processes}
\author{Miraine Dávila Felipe \and Amaury Lambert}
\address{Laboratoire de Probabilités et Modèles Aléatoires, Sorbonne University, UPMC Univ Paris 06, CNRS,
Case courrier 188, 4, Place Jussieu, 75005 Paris, France.}
\address{Center for Interdisciplinary Research in Biology (CIRB), Collège de France, CNRS, INSERM, PSL Research University,
11, place Marcelin Berthelot, 75005 Paris, France.}

\begin{document}

\selectlanguage{english}

\maketitle

\tableofcontents

\thispagestyle{empty}

\begin{abstract}

We consider a spectrally positive Lévy process $X$ that does not drift to $+\infty$, viewed as coding for the genealogical structure of a (sub)critical branching process, in the sense of a contour or exploration process \cite{GaJa98,Lam10}.
We denote by $I$ the past infimum process defined for each $t\geq 0$ by $I_t\coloneqq \inf_{[0,t]} X$ and we let $\gamma$ be the unique time at which the excursion of the reflected process $X-I$ away from 0 attains its supremum. We prove that the pre-$\gamma$ and the post-$\gamma$ subpaths of this excursion are invariant under space-time reversal, which has several other consequences in terms of duality for excursions of Lévy processes.
It implies in particular that the local time process of this excursion is also invariant when seen backward from its height.  
As a corollary, we obtain that some (sub)critical branching processes such as the (sub)critical Crump-Mode-Jagers (CMJ) processes and the excursion away from 0 of the critical Feller diffusion, which is the width process of the continuum random tree, are invariant under time reversal from their extinction time.
\end{abstract}

\pagestyle{headings}

\section{Introduction}

There exist several links between spectrally positive Lévy processes (SPLP) and branching processes that have been known and exploited for a few decades already. We can find their origin in the seminal works of Lamperti \cite{Lam67}, where it is shown that there exists a one-to-one correspondence, via a random time change (the so-called Lamperti transformation) between continuous state branching processes (CSBP) and possibly killed SPLP \cite{CaLaUB09,LaSiZw13}. 
Lévy processes also provide a suitable way of coding the genealogical structure of branching processes, through exploration or contour processes \cite{GaJa98,Lam10,LaUB16}.
Additionally, Ray-Knight type theorems link the local time processes of SPLP to the width processes of branching populations \cite{PaWa11,LaUB16}.  
\medskip

Here we consider a Lévy process with no negative jumps $X=(X_t,t \geq 0)$, with probability distribution $P_x = P(\cdot\vert X_0=x)$ and with Laplace exponent $\psi$, defined by 
\[E_0[\exp(-\lambda X_t)]= \exp(t\psi (\lambda)).
\]
Thanks to the Lévy-Kinchin formula, $\psi$ can be expressed as follows for any $\lambda\geq 0$,
\begin{align}\label{eq:lap-exp}
\psi(\lambda) = \alpha \lambda + \beta \lambda^2 + \int\limits_0^\infty \left(\e^{-\lambda r}-1+\lambda r \1_{r<1}\right)\Pi(\der r), 
\end{align}
where $\alpha \in\R$, $\beta\geq 0$ is called the \textit{Gaussian coefficient} and $\Pi$ is a $\sigma$- finite measure on $(0,\infty)$, called the \textit{Lévy measure}, satisfying $\int_{(0,\infty)}(r^2\wedge 1)\Pi(\der r)<\infty$. 
The paths of $X$ have finite variation a.s. if and only if $\beta=0$ and $\int_{(0,1]} r\Pi(\der r)<\infty$. Otherwise they have infinite variation a.s.

We consider the process reflected at its infimum $X-I$, where for each $t\geq 0$ we denote $I_t \coloneqq \inf_{[0,t]}X$.
A result due to Rogozin \cite{Rog66} states that for SPLP processes $0$ is always regular for $(-\infty,0)$, so is also regular for itself for the reflected process (and it is regular for $(0,+\infty)$ if and only if $X$ has infinite variation paths a.s.).
We know from general theory for Markov process that there exists a local time at 0 for $X-I$, here denoted by $(L_t, t \geq 0)$ that can be defined as the unique (up to a multiplicative constant) adapted additive functional that grows exactly on the zeros of $X-I$.
Furthermore, the fact that $X$ has no negative jumps entails that $-I$ satisfies these conditions, so it is an explicit local time for the reflected process.
Then, its right-continuous inverse
\begin{align*}
\tau_t \coloneqq \inf \left\{s>0: -I_s>t \right\}
\end{align*}
is the same as $T_{-t}= \inf\{s\geq 0: \ X_s<-t\}$, the first hitting time of $(-\infty,-t)$. It is a (possibly killed) subordinator whose jumps coincide exactly with the excursion intervals of $X-I$ so it represents the appropriate time scale for the so-called \textit{excursion process}, that we will now describe.

We let $\mathcal{E}$ be the space of real-valued càdlàg functions with finite lifetime $V\in[0,\infty)$, and we denote by $\partial$ a topologically isolated state, called \textit{cemetery point}. 
Define the excursion process $\epsilon = (\epsilon_t, 0< t\leq -I_{\infty})$, taking values in $\mathcal{E}\cup \partial$ as follows 
\begin{align*}
\epsilon_t \coloneqq \left\{
\begin{array}{ll}
		\left( (X_{\tau_{t-} + s} - I_{\tau_{t-}}, \ 0\leq s\leq \tau_t-\tau_{t-}\right),  & \text{if} \ \tau_t-\tau_{t-}>0, \\
		\partial, & \text{if } \ \tau_t-\tau_{t-}=0, \ \text{or} \ t=\infty,
	\end{array}
 \right.  \ \text{for} \ t\geq 0.
\end{align*}
Then according to Itô's excursion theory, $(t,\epsilon_t)_{t\geq 0}$ is a Poisson point process, possibly stopped at the first excursion with infinite lifetime (which arrives only and a.s. when $X$ drifts to $+\infty$). Its intensity is $\der t \ \underline{n}(\der \varepsilon)$, where $\underline{n}$ is a measure on $\mathcal{E}$ called the \textit{excursion measure}.
We refer to \cite[Chapter~IV]{Ber96} for further details.

From now on, we assume that $X$ is (sub)critical, meaning it does not drift to $+\infty$, which is equivalent to $\psi^\prime(0+)\geq 0$. 

Let $\varepsilon$ be the generic excursion of $X-I$ away from $0$ and $\gamma$ the first instant at which this excursion attains its supremum, that is
\begin{align*}
\gamma = \gamma(\varepsilon) = \inf\{s>0: \varepsilon_s=\overline{\varepsilon}_s\},
\end{align*}
where $\overline{\varepsilon}_s = \sup_{[0,s]}\varepsilon$.
Define the space-time-reversal transformation $\rho$ for any excursion $\omega\in \mathcal{E}$, as $\rho\circ \omega \coloneqq (\omega_{V-}-\omega_{(V-s)-}, 0\leq s\leq V)$, where by convention $\omega_{0-}=\omega_{0}$.  We also call it \textit{rotation}.
We are interested in the disintegration of $\varepsilon$ at $\gamma$. 
We prove that the pre-supremum and the post-supremum processes, denoted respectively  $\overset{\leftarrow}{\varepsilon}=k_\gamma\circ \varepsilon = (\varepsilon_s,0\leq s\leq\gamma)$ and $\overset{\rightarrow}{\varepsilon} = \theta^\prime_\gamma\circ\varepsilon:=(\varepsilon_{\gamma+s}-\varepsilon_\gamma,0\leq s\leq V-\gamma)$, are both invariant for this space-time-reversal transformation under the measure $\underline{n}$.
Moreover, this results implies the following theorem, for which we need first to define the functional $\chi:\mathcal{E}\rightarrow\mathcal{E}$ as
\begin{align*}
\chi \left( \varepsilon \right) \coloneqq \left[\rho\left( k_\gamma\circ \varepsilon\right) , \rho\left(\theta^\prime_\gamma\circ \varepsilon \right)+\varepsilon_\gamma \right],
\end{align*}
where for any two elements $\omega_1,\omega_2\in\mathcal{E}$, $[\omega_1,\omega_2]$ stands for their concatenation.

\begin{restatable}{theo}{thmain}
\label{th:main}
For every measurable functional $F:\mathcal{E}\rightarrow \R_+$ we have
\begin{align*}
\underline{n}\left( F \right) = \underline{n}\left( F\circ \chi \right)
\end{align*}
\end{restatable}
\medskip

Williams \cite{Will74} studied the decomposition of the generic excursion $X-I$ at its maximum for Brownian motion, showing it consists in two independent Bessel processes of dimension 3, started at 0, running to encounter each other, and killed upon hitting the same independent random variable (see e.g. \cite[Chapter~XII]{ReYo91} or \cite[Section~5]{Cha94}). This result has been generalized to SPLP by Chaumont \cite{Cha96}, and for general Lévy processes by Duquesne \cite{Duq03}. The law of each, the pre-$\gamma$ and post-$\gamma$ subpaths is characterized in \cite{Cha96} and \cite{Duq03} in terms of the law of $X$ \textit{conditioned to stay positive}, denoted $P^\uparrow$. Our study differs from these works since here we show that this distributions are invariant under \textit{space-time-reversal}.
Rather our results provide in passing properties of reversal invariance under the law $P^\uparrow$ that we state at the end of Section~\ref{sec:results}.
For further details on path decomposition theorems for Lévy processes at the overall maximum, minimum and other random times, we refer to \cite{Mil77a,Mil77b,GrPi80,Don07}.

More recently, in \cite{AbDe09,HeDe13} the authors also establish Williams decompositions under the excursion measure for the exploration process associated with the Lévy continuum random tree and super-processes with a spatially non-homogeneous quadratic branching mechanism. Several properties of these branching processes are then derived from these decompositions, such as a closed formula for the probability of hitting zero for a CSBP with immigration. In \cite{HeDe13}, the Q-process is obtained by looking at the super-process from the root and letting the extinction time tend to infinity. Moreover, an equivalent of the Esty time reversal from \cite{KlRoSa07} is given in a continuous setting.

Another related result is obtained by Miermont in \cite{Mie01} concerning a similar decomposition via the Vervaat's transform. 
Proposition 1 in \cite{Mie01} applied to SPLP of infinite variation not drifting to $+\infty$ and having bi-continuous marginal densities w.r.t. to Lebesgue measure, has the following implications: the excursion above the infimum conditioned to have a duration equal to $l$ is well defined, and if we cut this excursion at a uniform point $\upsilon$ and we concatenate the post-$\upsilon$ and the pre-$\upsilon$ subpaths (in this order), we get a Lévy bridge going from $0$ to $0$ in $l$ units of time (which is well defined under these hypotheses). Such a bridge is clearly invariant by rotation and hence, by conditioning respectively on the events $\{\upsilon<\gamma\}$ and $\{\upsilon>\gamma\}$, we could have obtained that the laws of $\overset{\leftarrow}{\varepsilon}$ and $\overset{\rightarrow}{\varepsilon} $ are also invariant by rotation.
This approach provides an alternative way of obtaining some of our results under the technical hypothesis of continuity on the marginal densities.

\medskip

A first consequence of  Theorem~\ref{th:main} is the invariance under time reversal of the local time process of the excursion $X-I$ away from 0.
To be more specific, define the local time process $(\Gamma(\varepsilon,r),r\geq 0)$ for the canonical excursion $\varepsilon\in \mathcal{E}$ as a Borel function satisfying
\begin{align}\label{eq:loc-time-inf-var}
\int\limits_{0}^{V(\varepsilon)}\phi\left(\varepsilon_s\right) \der s = \int\limits_{0}^{\infty} \Gamma\left(\varepsilon,r \right) \phi(r)	\der r,
\end{align}
for any continuous function $\phi$ with compact support in $[0,\infty)$.
These local time processes are known to exist and be unique in the infinite variation case, named occupation densities, see for instance \cite{Ber96}.
When $X$ has finite variation we can define an equivalent process, taking values in $\N\cup\{+\infty\}$, as the number of times the excursion hits level $r$, i.e. 
\begin{align}\label{eq:loc-time-fin-var}
\Gamma\left(\varepsilon,r \right)  = \sum_{0\leq s\leq V} \1_{\{\varepsilon_s=r\}}.
\end{align} 

Then we can state the following result.

\begin{restatable}{corollary}{cormain}
\label{cor-rev-loc-time}
The local time process of the excursions of $X-I$ away from 0 is invariant under time reversal, that is
\begin{align}\label{eq:rev-loc-time}
\left( \Gamma\left(\varepsilon,r \right), 0\leq r \leq \varepsilon_\gamma \right)  \overset{\text{d}}{=}  \left( \Gamma\left(\varepsilon, \varepsilon_\gamma-r \right), 0\leq r \leq \varepsilon_\gamma \right).
\end{align}
\end{restatable}

\medskip

As we previously pointed out, the main motivation that led us to look into this pathwise decomposition comes from \textit{branching processes}, that is stochastic processes with non-negative values satisfying the \textit{branching property}. This means that for any $x,y>0$, the process started at $x+y$ has the same distribution as a sum of two independent copies of itself, starting respectively at $x$ and $y$.
The simplest branching processes are those in discrete time and state space, the well-known Bienaymé-Galton-Watson (GW) processes \cite{AthNe72}. In the case of discrete time and continuous state-space, we use the term \textit{Jirina processes} as in \cite{Lam10}. For continuous time and discrete state space we speak of Crump-Mode-Jagers (CMJ) processes and finally the so-called \textit{continuous state branching process} (CSBP) for continuous time and state spaces. 
Our results concern mainly the latter two, so we will spend more time specifying their characteristics in Section~\ref{sec:applications}, but we refer to \cite{AthNe72,Jag75} for the general theory of branching processes.

\textit{Splitting trees} are random trees formed by discrete particles that behave independently from each other, have i.i.d. lifetime durations (possibly infinite), and give birth to i.i.d. copies of themselves during their lives (single births). 
The excursion of $X-I$ can be viewed, in the finite variation case, as the contour process of a (sub)critical splitting tree and then its local time process is a CMJ \cite{Lam10}. 
In the infinite variation case, under some mild assumptions, this excursion codes the genealogy of a \textit{totally ordered measured} (TOM) tree satisfying the splitting property, which are the continuum analogue of chronological trees in the setting of real trees, as it is shown in \cite{LaUB16}. 
Hence, Theorem \ref{th:main} leads in particular to the invariance under time reversal from its extinction time of the (sub)-critical CMJ process which is the local time of the SPLP characterized by \eqref{eq:lap-exp}. The same holds for the excursion away from 0 of the \textit{critical Feller diffusion} (the CSBP with branching mechanism $\psi(\lambda) = \lambda^2$, see \cite{DuGa02} and Subsection~\ref{subsec:CRT} for the details on this definition), which is the width process of the continuum random tree, \cite{Ald91}.
This can be summarized in the following corollary of Theorem~\ref{th:main}.

\begin{cor}
\label{cor-rev-loc-time-bis}
The (sub)critical CMJ branching process and the excursion away from 0 of the critical Feller diffusion, are invariant under time reversal from their extinction time.
\end{cor}

Similar results concerning the duality by time-reversal of branching processes have been given in the litterature. In particular, in \cite{AlPo05} we can find a time-reversal invariance principle for the linear birth and death process in the critical case, when the process is conditioned on the number of individuals at the time of reversal to be equal to $n$. As suggested by the authors, the rescaled limit of the time-reversed process when $n\rightarrow\infty$, is the Feller branching diffusion. This suggests an alternative way of obtaining the second result in the previous Corollary.
See also \cite{Est75} and more recently \cite{AlRo02,KlRoSa07,HeDe13} for the treatment of the reverse of Galton-Watson processes and specifically the \textit{Esty time reversal}, which is the limit obtained by conditioning a GW process in negative time upon entering the state 0 (extinction) at time 0 and starting in the state 1 at time $-n$, when $n$ tends to $+\infty$.
We also refer to \cite{BiDe16} for a time reversal property for the ancestor counting process of a stationary CSBP with sub-critical quadratic branching mechanism.

This paper is a follow-up to \cite{DaLa15}, where we have obtained a property of invariance under time-reversal, from a deterministic time $T$, for the population size process of certain random forests. The latter are defined as a sequence of independent splitting trees stopped at the first tree surviving up to time $T$. 
In \cite{DaLa15} we focused on the time-reversal from a deterministic time $T$ whereas here we are interested in the same property from the extinction time of the process.
It is worth stressing that, besides the implications concerning branching processes, some of our lemmas are interesting in their own right since they provide some invariance results for subpaths of SPLP.

\medskip

The paper is organized as follows. In a short Section~\ref{sec:preliminaries} we introduce some preliminary notions and notation. It is followed up by Section~\ref{sec:results}, which contains our main results on the path decompositions of SPLP reflected at their infimum under the excursion measure. In Section~\ref{sec:applications} we recall some notions linking SPLP to branching processes and give the main implications of our results in the context of the latter. Finally Section~\ref{sec:proofs} is devoted to completing the remaining demonstrations.

\section{Preliminaries}\label{sec:preliminaries}

\subsubsection*{Basic notation}
Let $\mathcal{B}(\R)$ denote the Borel $\sigma$-field of $\R$. 
Consider the space $\D(\R_+,\R)$ (or simply $\D$) of càdlàg functions $\omega$ from $\R_+$ into the measurable space $(\R,\mathcal{B}(\R))$ endowed with Skorokhod topology \cite{JaShi03}. Denote the corresponding Borel $\sigma$-field by $\mathcal{B}(\D)$. 
Define the \textit{lifetime} of a path $\omega \in \D$ as $V=V(\omega)\coloneqq \inf\{t\geq 0: \omega(s)=\omega(t), \forall s\geq t\}$, with the convention $\inf \emptyset = \infty$.
Here $\omega(t-)$ stands for the left limit of $\omega$ at $t\in\R_+$, $\Delta \omega(t)= \omega(t)-\omega(t-)$ for the size of the (possible) jump at $t\leq V$ and we adopt the usual convention $\omega(0-)=\omega(0)$. 
The subspace of functions in $\D$ with finite lifetime is denoted $\mathcal{E}$ as in the introduction. We will consider invariably $\omega\in  \mathcal{E}$ as a function defined on $[0, V(\omega)]$ or as its prolongation to $\R_+$ obtained by stopping. 

We consider stochastic processes, on the probability space $(\D,\mathcal{B}(\D),P)$, say $X=\left(X_t,t\geq 0\right)$, also called the coordinate process, having $X_t=X_t(\omega)=\omega(t)$. In particular, we will consider only processes with no negative jumps, that is such that $\Delta X_t \in\R_+$ for every $t\geq 0$.
The canonical filtration is denoted by $(\mathscr{F}_t)_{t\geq 0}$ and we let $\mathcal{P}(E)$ be the collection of all probability measures on any space $E$. We use the notation  $P_x(X\in\cdot ) = P \left( X\in\cdot \middle\vert X_0 = x\right)$. In the absence of subscript the process is considered to start at $0$ a.s.

Define by $T_A \coloneqq \inf\{t> 0: X_t\in A\}$, the first hitting time of the set $A\in\mathcal{B}(\R)$, with the conventions $T_0=T_{\{0\}}$, and for any $x>0$, $T_{-x} =T_{(-\infty,-x)}$, $T_x =T_{(x,+\infty)}$.
Note that in general  $ T_{\{\pm x\}} \neq T_{\pm x}$,  for $x>0$. However, since $X$ has no negative jumps, $X$ is a.s. continuous at $T_{-x}$ for all $x>0$, and then it holds that $ T_{\{-x\}} = T_{-x} = T_{(-\infty,-x)}$ a.s.

As usual, for real valued functions, $\Vert \cdot\Vert_\infty$ stands for the uniform norm and $\Vert \omega \Vert_T \coloneqq \sup_{[0,T]}\vert \omega \vert$ for the supremum over $[0,T]$.

\subsubsection*{Some path transformations of càdlàg functions}

In this subsection we will define some families of operators on the space of càdlàg functions $\omega\in\D$:
\begin{itemize}
\item the classical \textit{shift operators}, $\theta_s, s\in \R_+$, defined by
\[
[\theta_s(\omega)]_t\coloneqq \omega_{s+t}, \qquad \forall t\in \R_+
\]
\item the non-standard shift operators, $\theta^\prime_s, s\in \R_+$, defined by
\[
[\theta^\prime_s(\omega)]_t\coloneqq \omega_{s+t}-\omega_{s}, \qquad \forall t\in \R_+
\]
\item the \textit{killing operators}, $k_s$, $s\in\R_+$ , defined by
\[
[k_s(\omega)]_t \coloneqq \left\{
	\begin{array}{ll}
		\omega_t, & \mbox{if } t < s\ \\
		\omega_s, & \mbox{otherwise}
	\end{array}
\right.
\]
the killing operator can be generalized to killing at random times, for instance $k_{T_A}(X) = k_{T_A(X)}(X)$, denotes the process $X$, \textit{killed at the first passage into $A$}. It is easy to see that if $X$ is a Markov process, so is $k_{T_A}(X)$. We set 
\[
k_0 (\omega) \equiv \omega_0.
\]
What we call killed path here, is more commonly denominated \textit{stopped} path. The difference is that killing usually refers to the path being sent to an isolated state after the killing time, whereas here it remains constant to a real value. We highlight our interest in keeping track of the final jump of the functions we study, which justifies this choice.

\item the \textit{space-time-reversal} mapping $\rho_s$, $s\in\R_+^*$, as
\[
[\rho_s(\omega)]_t \coloneqq 
\left\{
	\begin{array}{ll}
		\omega_s - \omega_{(s-t)-} & \quad \forall t\in[0,s]\ \\
		\omega_s - \omega_{0} & \mbox{if } t > s
	\end{array}
\right.
\]
and when $V(\omega)<+\infty$, we call \textit{rotation}, denoted simply by $\rho$, the space-time-reversal operator at the lifetime  of a path, that is $\rho=\rho_V$.
Notice that $[\rho_s(\omega)]_0=\Delta \omega_s$ (possibly $\neq 0$).
\end{itemize}
The notations $P\circ \theta_s^{-1}$, $P\circ k_s^{-1}$ and $P\circ \rho^{-1}$ stand for the law of the shifted, killed and space-time-reversed processes when $P$ is the law of $X$.

\medskip

For a sequence of functions in the same state space, say $(\omega_i)_{i\geq 1}$ with lifetimes $(V_i, i\geq 1)$, we define a new process by the concatenation of the terms of the sequence, denoted by
\[
 [\omega_1,\omega_2,\dots]
\]
where the juxtaposition of terms is considered to stop at the first element with infinite lifetime. For instance, if $V(\omega_1)<+\infty$ and $V(\omega_2)=+\infty$, then for every $n\geq 2$
\[
 [\omega_1,\omega_2, \ldots, \omega_n]_t = 
\left\{
	\begin{array}{ll}
		\omega_{1,t}& \mbox{if } 0\leq t\leq V(\omega_1) \\
		\omega_{2,t-V(\omega_1)} & t > V(\omega_1)
	\end{array}
\right..
\]
Notice that a concatenation of càdlàg functions thus defined, might not be a càdlàg function, since, for instance, in the case where $n=2$, the first function might end with a jump, so the new function $[\omega_1,\omega_2]$ will be càdlàg only if $\omega_{1,V_1} = \omega_{2,0}$. This is always the case  in our applications, that is why we choose to concatenate functions in this less usual way, which has the property of recording the final jump of each path.

\subsubsection*{Skorokhod topology}

As mentioned before, we consider the space of càdlàg functions $\D$ and the subset of excursions $\mathcal{E}$ (paths with finite lifetime), to be endowed with the topology induced by Skorokhod's topology, which makes $\D$ a Polish space. We refer to \cite{JaShi03} for further details on this topology, which can be characterized as follows: a sequence $(\varepsilon_n)$ on $\mathcal{E}$ converges to $\varepsilon$ when $n\rightarrow\infty$, if and only if there exists a sequence $(\lambda_n)$ of \textit{changes of time} (continuous, strictly increasing functions, with $\lambda_n(0)=0$ and $\lambda_n(t)\uparrow \infty$ when $t\uparrow \infty$), such that $\Vert \lambda_n - \textrm{Id}\Vert_\infty \rightarrow 0$ and $\Vert \varepsilon_n\circ \lambda_n -\varepsilon\Vert_T \rightarrow 0$ for all $T\geq 0$.
The space of continuous, bounded functions from $\mathcal{E}$ into $\R_+$ with respect to the Skorokhod topology, will be denoted by $\mathcal{C}_b(\mathcal{E},\R_+)$.

\subsubsection*{Properties of the Laplace exponent and scale function of a SPLP}
The Laplace exponent given by \eqref{eq:lap-exp} is infinitely differentiable, strictly convex (when $\Pi \not \equiv 0$ or $\beta\neq 0$), $\psi(0)=0$ and $\psi(+\infty)=+\infty$.
Let $\eta\coloneqq \sup\{\lambda\geq 0: \psi(\lambda)=0\}$. Then we have that $\eta=0$ is the unique root of $\psi$, when  $\psi^\prime(0+)\geq 0$. Otherwise the Laplace exponent has two roots, 0 and $\eta>0$. 
It is known that for any $x>0$,
\[P_x\left(T_0<+\infty\right)=\e^{-\eta x}.
\]
More generally, there exists a unique continuous increasing function $W:[0,+\infty)\rightarrow[0,+\infty)$, called the \textit{scale function}, characterized by its Laplace transform,
\[
\int\limits_0^{+\infty} \e^{-\lambda x}W(x) \der x = \dfrac{1}{\psi(\lambda)}, \qquad \lambda>\eta,
\]
such that for any $0<x<a$,
\begin{equation} \label{eq:exit-two-side}
P_x\left(T_0<T_a\right) = \dfrac{W(a-x)}{W(a)}.
\end{equation}

\subsubsection*{Time-reversal duality for Lévy processes}

One of the key ingredients of our results is the duality property under time-reversal  of Lévy processes (see \cite[Chapter~II]{Ber96} for details). Roughly speaking, it states that if a path is space-time-reversed at a finite time horizon, the new path has the same distribution as the original process.
We will use the following formulation subsequently: for every fixed $t>0$ and every non-negative measurable function $F$ we have 
\begin{eqnarray}
&&E \left[ F\left(k_t\circ X \right)\right] = E \left[ F\left(\rho \circ(k_t\circ X)\right)\right].
\label{eq:duality}
\end{eqnarray}
By integrating over $t$, this result is still valid if the process is killed at an independently random finite time.

\section{Main results}\label{sec:results}

Throughout this section $X$ denotes a SPLP with Lévy measure $\Pi$ on $(0,+\infty)$, whose Laplace exponent denoted by $\psi$ is defined by \eqref{eq:lap-exp}.
As in the preliminaries, we let $P_x$ denote the law of the process conditioned on $X_0=x$. We assume $\psi^\prime(0+)\geq 0$, meaning we are in the (sub)critical regime.
Let $S_t\coloneqq \sup\{X_s, 0\leq s\leq t\}$ and $I_t\coloneqq \inf\{X_s, 0\leq s\leq t\}$ be the running supremum and the running infimum of the Lévy process $X$.

\subsection*{Pre-supremum processes}

We recall that $\underline{n}$ denotes the excursion measure of the process $X-I$ away from $0$. Let $g_t$ and $d_t$ be the left and right-end points of the excursion straddling $t$, denoted $e_t$, that is,
\begin{equation*}
e_t \coloneqq \left( X_{g_t+s}-I_t, 0\leq s\leq d_t-g_t\right).
\end{equation*}
Note that $V(e_t)=d_t-g_t$.

For any excursion $\varepsilon$ and any $s\in\R_+$ define its supremum $\overline \varepsilon_{s} \coloneqq \sup_{[0,s]} \varepsilon $ and the first instant where the supremum is attained on the interval $[0,s]$ , that is
\begin{align*}
& \gamma(s) = \gamma(s,\varepsilon) \coloneqq \argmax_{[0,s]} \varepsilon = \inf\left\{s^\prime\in[0,s]:\varepsilon(s^\prime)=\overline \varepsilon_s\right\}.
\end{align*}
A classical result \cite{Mil77a,Duq03} ensures that this instant is unique $P$-a.s. and $\underline{n}$-a.e. thanks to the regularity of 0 for $(-\infty,0)$. In general when using $\gamma$ and $V$, the dependence on the excursion under focus will be omitted unless there is a risk of confusion.

Define in a similar way, for the process $X$,
\begin{align*}
&\overline{\sigma}_t(X) \coloneqq \argmax_{[0,t]} X, \quad \text{and}
\\
&\underline{\sigma}_t(X) \coloneqq \arginf_{[0,t]} X = \sup\left\{s'\in[0,t]:I_{s'}=X_{s'-}\right\}.
\end{align*}
Note that  this last instant is also unique a.s. and that it is a simple `$\argmin$' in the infinite variation case.

Let $\omega\in \mathcal{E}$ be any path with finite lifetime. We are interested in the trajectories where the infimum is attained before the maximum, so let us define the event
\[
\A(\omega)\coloneqq \left\{ \arginf_{}\omega < \argmax_{} \omega\right\}.
\]
We will be interested in particular in $\A(k_t\circ X)$ where
\[
\A(k_t\circ X)= \left\{\underline{\sigma}_t(X)< \overline{\sigma}_t(X)\right\} 
= \left\{S_{g_t}-X_{g_t}< \sup_{(0,t-g_t)} e_t\right\}.
\]

We can now state the next result. 

\begin{prop}\label{prop:pre-sup}
The pre-supremum process of the excursion of $X-I$ away from zero is invariant under time reversal, that is, for any measurable functional $h:\mathcal{E}\rightarrow \R_+$,
\begin{align}
\underline n\left(   h\left( k_{\gamma(V)}\circ \varepsilon\right) \right)=\underline n\left(   h\circ\rho \left( k_{\gamma(V)}\circ\varepsilon \right) \right). \label{eq:id-n}
\end{align}
\end{prop}
\medskip

This result is based on the following lemmas, for which we need first to define the set $\mathcal{H}$ of functions of \textit{exponential type} in the lifetime of excursions. That is, measurable functions $f:\mathcal{E}\rightarrow \R_+$, such that there exist two non-negative constants $c$ and $C$ such that $f(\varepsilon)\e^{-cV(\varepsilon)}\leq C$, for every $\varepsilon\in\mathcal{E}$.

\begin{lemma}\label{lemma:continuity-on-s}
For any functional $f\in \mathcal{H}$, the following functions are right-continuous at every $s>0$
\begin{enumerate}[label=(\roman*)]
\item
$\underline n\left(   f\left( k_{\gamma(s)}\circ \varepsilon\right) \1_{\{s<V\}}\right),$
\item 
$\underline n\left(   f\left( k_{\gamma(V)}\circ \varepsilon\right) \1_{\{\gamma(V)<s<V\}}\right)$.
\end{enumerate}
\end{lemma}
\begin{proof}
See Section~\ref{sec:proofs}.
\end{proof} 

\begin{lemma}\label{lemma:reverse-pre-sup}
For any functional $f\in \mathcal{H}$ and every $s\geq 0$, we have the following identities
\begin{enumerate}[label=(\roman*)]
\item
\begin{align}
\underline n\left(   f\left(  k_{\gamma(s)}\circ\varepsilon\right)\1_{\{s<V\}} \right)=\underline n\left(   f\circ\rho\left( k_{\gamma(s)}\circ \varepsilon\right) \1_{\{s<V\}} \right), \label{eq:id-n1}
\end{align}
\item 
\begin{align}
\underline{n}\left( f( k_{\gamma(V)}\circ\varepsilon) \1_{\{\gamma(V)<s<V\}} \right) = \underline{n}\left( f\circ \rho ( k_{\gamma(V)}\circ\varepsilon) \1_{\{\gamma(V)<s<V\}}\right). \label{eq:id-nV}
\end{align}
\end{enumerate}
\end{lemma}

\begin{proof}
Let us define the following functional
\begin{align*}
F_1\left(k_t\circ X\right) = f\left(k_{\gamma(V)} \circ \theta^\prime_{\underline{\sigma}_t(X)}\circ k_t(X)  \right)\1_{\left\{\left(\overline{\sigma}_t-\underline{\sigma}_t\right)(X)>0\right\}} g\left(X_{\overline{\sigma}_t}-X_{\underline{\sigma}_t-},X_t-X_{\underline{\sigma}_t-}\right),
\end{align*}
where $f, g$ are also non-negative measurable functions.
It is not hard to see that a.s. $\A(k_t\circ X) = \{(\overline{\sigma}_t-\underline{\sigma}_t)(X)>0\} = \A(\rho (k_t\circ X))$
since $\underline\sigma_t(\rho (X\circ k_t)) = t-\overline\sigma_t(k_t\circ X)$ and $ \overline\sigma_t(\rho(k_t\circ X)) = t-\underline\sigma_t(k_t\circ X)$. Hence, the duality \eqref{eq:duality} applied to $F_1$ gives that
\begin{multline}
E \left[ f\left(X_{\underline{\sigma}_t+s} - X_{\underline{\sigma}_t-}, 0\leq s\leq \overline{\sigma}_t - \underline{\sigma}_t \right)\1_{\{\underline{\sigma}_t<\overline{\sigma}_t\}}
g\left(X_{\overline{\sigma}_t}-X_{\underline{\sigma}_t-},X_t-X_{\underline{\sigma}_t-}\right) \right]
\\
= E \left[ f\left(X_{\overline{\sigma}_t} - X_{(\overline{\sigma}_t-s)-}, 0\leq s\leq \overline{\sigma}_t  - \underline{\sigma}_t \right)\1_{\{\underline{\sigma}_t<\overline{\sigma}_t\}}
g\left(X_{\overline{\sigma}_t}-X_{\underline{\sigma}_t-},X_{\overline{\sigma}_t}\right)
\right].
\nonumber
\end{multline} 
Notice that $\underline\sigma_t$ is the left-end point of $e_t$, the excursion straddling $t$, this point is denoted $g_t$. Similarly, $\overline\sigma_t$ is the point where this excursion attains its maximum before $t$. This implies that $P$-a.s.  on $\A(k_t\circ X)$, we have $k_{\gamma(t-g_t)}\circ e_t  = \left(X_{\underline{\sigma}_t+s} - X_{\underline{\sigma}_t-}, 0\leq s\leq \overline{\sigma}_t - \underline{\sigma}_t \right)$ and $\rho\left( k_{\gamma(t-g_t)}\circ e_t\right) = \left(X_{\overline{\sigma}_t} - X_{(\overline{\sigma}_t-s)-}, 0\leq s\leq \overline{\sigma}_t  - \underline{\sigma}_t \right)$, which allows us to write the preceding identity as follows
\begin{subequations}
\label{eq:reverse-i}
\begin{align}
\label{eq:reverse-i-lhs}
&E \left[ f\left(k_{\gamma(t-g_t)}\circ e_t \right) \1_{\A(k_t\circ X)} g\left(e_t(\gamma(t-g_t)),e_t(t-g_t)\right)
\right] 
\\
\label{eq:reverse-i-rhs}
& \qquad = E \left[ f\left(\rho \left(k_{\gamma(t-g_t)}\circ e_t\right)\right)\1_{\A(\rho\circ (k_t\circ X))} g\left(e_t(\gamma(t-g_t)),e_t(\gamma(t-g_t)) + I_t\right)
\right]. 
\end{align}
\end{subequations}

We will first develop the left-hand side of this equation. First, rewrite $\A(k_t\circ X)$ as $\{S_{ g_t}-X_{g_t} < \max_{(0,t-g_t)}(e_t)\}$ and instead of stopping the process at $t$, we kill it at an exponential independent rate, or equivalently, we integrate \eqref{eq:reverse-i-lhs} against $q\e^{-q t}$, giving
\begin{align*}
\int\limits_0^{+\infty} q \e^{-qt} \der t \ E\left[ f\left( k_{\gamma(t-g_t)}\circ e_t \right) \1_{\{S_{g_t}-X_{g_t} < e_{t}\left(\gamma(t-g_t)\right)\}} g \left(e_{t}\left(\gamma(t-g_t)\right), e_{t}\left(t-g_t\right)\right)\right].
\end{align*}

As in the introduction, we let $(\tau_u)_{u\geq 0}$ denote the inverse of the local time at $0$ of the process $X-I$, and by $\epsilon_u$ the excursion starting at $\tau_{u-}$. If we exchange the expectation and the integral in the preceding equation (Fubini's theorem), we can express the quantity inside the expectation as a sum taken over all the excursion intervals of $X-I$ away from 0
\begin{align*}
&\int\limits_0^{+\infty} q \e^{-qt} \der t 
E\Biggl[ \sum\limits_{u:\Delta \tau_u>0} \1_{\{\tau_{u-}<t\leq \tau_u\}} f\left( k_{\gamma(t-\tau_{u-})} \circ \epsilon_u\right) 
\1_{\{S_{\tau_{u-}} - X_{\tau_{u-}} < \epsilon_u \left(\gamma(t-\tau_{u-})\right)\}} \biggr.
\\
& \qquad \qquad \qquad \qquad \qquad \qquad \qquad \qquad \qquad \quad \times
\Biggl. g \left(\epsilon_u \left(\gamma(t-\tau_{u-})\right),\epsilon_u\left(t-\tau_{u-} \right) \right) 
\Biggr]
\\
&=E\Biggl[  \sum\limits_{u:\Delta \tau_u>0}  \int\limits_0^{\infty} q \e^{-qt} \der t \1_{\{\tau_{u-}<t\leq \tau_u\}} f\left( k_{\gamma(t-\tau_{u-})} \circ \epsilon_u \right) 
\1_{\{S_{\tau_{u-}} - X_{\tau_{u-}} < \epsilon_u \left(\gamma(t-\tau_{u-})\right)\}} \biggr.
\\
& \qquad \qquad \qquad \qquad \qquad \qquad \qquad \qquad \qquad \quad \times
\Biggl. g \left(\epsilon_u \left(\gamma(t-\tau_{u-})\right),\epsilon_u\left(t-\tau_{u-} \right) \right) 
\Biggr]. 
\end{align*}
We have applied Fubini's theorem once more for the last step.
By the change of variable $s=t-\tau_{u-}$, taking again the expectation and applying the compensation formula \cite[Chapter IV]{Ber96} we get
\begin{subequations}
\label{eq:E-phi}
\begin{align}
\label{eq:E-phi-lhs}
&E\left[ \sum\limits_{u:\Delta \tau_u>0}  \int\limits_0^{\Delta \tau_u} q \e^{-q\tau_{u-}}\e^{-qs} \der s \ f\left( k_{\gamma(s)} \circ \epsilon_u\right) 
\1_{\{S_{\tau_{u-}} - X_{\tau_{u-}} < \epsilon_u \left(\gamma(s)\right)\}} 
g \left(\epsilon_u \left(\gamma(s)\right),\epsilon_u\left(s\right) \right) 
\right]
\\
& \ = E\left[ \int\limits_0^{+\infty}\der u\ \e^{-q\tau_u} \int \underline n(\der \varepsilon) \int\limits_0^{V(\varepsilon)} \der s \ q \e^{-qs}  f\left( k_{\gamma(s)} \circ \varepsilon \right) 
\1_{\{S_{\tau_{u}} - X_{\tau_{u}} < \varepsilon_{\gamma(s)}\}} 
g \left(\varepsilon_{\gamma(s)},\varepsilon_s \right) 
\right] \nonumber
\\
\label{eq:E-phi-rhs}
&\qquad = E\left[ \int\limits_0^{+\infty}\der u\ \e^{-q\tau_u} \varphi\left( S_{\tau_u} - X_{\tau_u} \right)\right],
\end{align}
\end{subequations}
where $\varphi(x)= \underline{n}\left( \int_0^{V(\varepsilon)} \der s \ q\e^{-qs} f(k_{\gamma(s)} \circ \varepsilon)\1_{\{x<\varepsilon_{\gamma(s)}\}} g \left(\varepsilon_{\gamma(s)},\varepsilon_s \right) \right)$.

Define $G_q\coloneqq \arginf_{(0,\mathbbm{e}_q)} X$, where $\mathbbm{e}_q$ is exponentially distributed with parameter $q$ and is independent of $X$. Then we can apply again the compensation formula to expand the last expression as follows
\begin{align*}
E\left[ \varphi\left(S_{G_q} - X_{G_q}\right) \right] &= E\left[ \sum\limits_{u:\Delta \tau_u>0} \1_{\{\tau_{u-}<\mathbbm{e}_q<\tau_u\}} \varphi\left(S_{\tau_{u-}} - X_{\tau_{u-}}\right) \right] 
\\
& \quad =E\left[ \sum\limits_{u:\Delta \tau_u>0} \left(\e^{-q\tau_{u-}}-\e^{-q\tau_{u}}\right) \varphi\left(S_{\tau_{u-}} - X_{\tau_{u-}}\right)  \right]  
\\
& \qquad =E\left[ \int\limits_0^{+\infty} \der u \ \e^{-q\tau_u} \varphi\left(S_{\tau_{u}} - X_{\tau_{u}}\right) \right] \underline{n}\left(1 - \e^{-qV}\right). 
\end{align*}
So we get that \eqref{eq:E-phi-lhs} is equal to
\begin{align}
\dfrac{E\left[ \varphi\left(S_{G_q} - X_{G_q}\right)\right]}{\underline{n}\left(1 - \e^{-qV}\right)}, \label{eq:E-phi-1}
\end{align}
and by applying Fubini's theorem one more time, and replacing $\varphi$ by its expression this is equal again to
\begin{align*}
& \dfrac{q}{\underline{n}\left(1 - \e^{-qV}\right)} E\left[ \underline n\left( \int\limits_0^V \der{s} \ \e^{-qs} f\left( k_{\gamma(s)}\circ \varepsilon \right)  \1_{\{Y_q<\varepsilon_{\gamma(s)}\}} g \left(\varepsilon_{\gamma(s)},\varepsilon_s \right) \right)\right]  
\\
&\quad = \dfrac{q}{\underline{n}\left(1 - \e^{-qV}\right)} \underline n\left( \int\limits_0^V \der s \ \e^{-qs} f\left( k_{\gamma(s)}\circ \varepsilon \right) P\left(Y_q<\varepsilon_{\gamma(s)}\right) g \left(\varepsilon_{\gamma(s)},\varepsilon_s \right) \right)
\\
&\qquad = \dfrac{q}{\underline{n}\left(1 - \e^{-qV}\right)} \int\limits_0^{\infty} \der s \ \e^{-qs} \underline n\left(  f\left( k_{\gamma(s)}\circ \varepsilon\right) \1_{\{s<V\}} P\left(Y_q<\varepsilon_{\gamma(s)}\right) g \left(\varepsilon_{\gamma(s)},\varepsilon_s \right) \right). 
\end{align*}
where $Y_q$ stands for an independent random variable, distributed as $S_{G_q} - X_{G_q}$.
\medskip

We now go back to Equation~\eqref{eq:reverse-i} and apply all the above arguments to \eqref{eq:reverse-i-rhs}. By choosing the same function $f$ and observing that $\varepsilon_{\gamma(s)} = \max_{[0,s]} \varepsilon = \max_{[0,s]} \rho \left(k_{\gamma(s)}\circ \varepsilon\right)$, we obtain that this is equal to the following expressions, analogous to \eqref{eq:E-phi} and \eqref{eq:E-phi-1},
\begin{align}
E\left[ \int\limits_0^{+\infty}\der u\ \e^{-q\tau_u} \widetilde\varphi\left( S_{\tau_u} - X_{\tau_u}, I_{\tau_u} \right)\right]
= \dfrac{E\left[ \widetilde\varphi\left(S_{G_q} - X_{G_q},I_{G_q}\right)\right]}{\underline{n}\left(1 - \e^{-qV}\right)}, \nonumber
\end{align}
where $\widetilde\varphi(x,y)= \underline{n}\left( \int_0^{V(\varepsilon)} \der s q\e^{-qs} f\circ \rho(k_{\gamma(s)}\circ \varepsilon)\1_{\{x<\varepsilon_{\gamma(s)}\}} g \left(\varepsilon_{\gamma(s)},\varepsilon_{\gamma(s)}+y\right) \right)$. 
Using the same arguments as before, and denoting by $(Y_q,Z_q)$ a pair distributed as $(S_{G_q}-X_{G_q}, I_{G_q})$, this is also equal to
\begin{align*}
&\dfrac{q}{\underline{n}\left(1 - \e^{-qV}\right)} E\left[ \underline n\left( \int\limits_0^V \der{s} \ \e^{-qs} f\circ\rho\left( k_{\gamma(s)}\circ \varepsilon\right)  \1_{\{Y_q<\varepsilon_{\gamma(s)}\}} g \left(\varepsilon_{\gamma(s)},\varepsilon_{\gamma(s)} +Z_q\right) \right)\right]  
\\
&= \dfrac{q}{\underline{n}\left(1 - \e^{-qV}\right)}  \int\limits_0^\infty \der{s} \ \e^{-qs} \underline n\left(  f\circ\rho\left( k_{\gamma(s)}\circ \varepsilon\right) \1_{\{s<V\}} E\left[ \1_{\{Y_q<\varepsilon_{\gamma(s)}\}} g \left(\varepsilon_{\gamma(s)},\varepsilon_{\gamma(s)} +Z_q\right) \right] \right),
\end{align*}
where it should be noted that the expectation is taken with respect to the law of $(Y_q,Z_q)$.
Finally, since the first term in the above product is the same for \eqref{eq:reverse-i-lhs} and \eqref{eq:reverse-i-rhs}, Equation~\eqref{eq:reverse-i} is equivalent to
\begin{align}
&\int\limits_0^{\infty} \der s \ \e^{-qs} \underline n\left(  f\left( k_{\gamma(s)}\circ \varepsilon\right) \1_{\{s<V\}} P\left(Y_q<\varepsilon_{\gamma(s)}\right) g \left(\varepsilon_{\gamma(s)},\varepsilon_s \right) \right) \nonumber 
\\
& = \int\limits_0^\infty \der{s} \ \e^{-qs} \underline n\left(  f\circ\rho\left( k_{\gamma(s)}\circ \varepsilon\right) \1_{\{s<V\}} E\left[ \1_{\{Y_q<\varepsilon_{\gamma(s)}\}} g \left(\varepsilon_{\gamma(s)},\varepsilon_{\gamma(s)} +Z_q\right) \right] \right).
\label{eq:reverse-ii}
\end{align}

\begin{enumerate}[label=(\roman*)]

\item In order to prove the first identity in the lemma we start by taking $g\equiv 1$. The probability  $P\left(Y_q<\varepsilon_{\gamma(s)}\right) $ is also a function of $ k_{\gamma(s)}\circ \varepsilon$, that is in addition strictly positive since $\varepsilon_{\gamma(s)}=0$ only if $\varepsilon=0$.
Therefore we can consider 
\begin{equation*}
f \left( k_{\gamma(s)}\circ \varepsilon\right) = \dfrac{h\left(k_{\gamma(s)}\circ \varepsilon\right) \e^{\alpha\gamma(s)}}{P\left(Y_q<\varepsilon_{\gamma(s)}\right)}, \ \ \forall s> 0,
\end{equation*}
where $h$ is a non-negative bounded function and $\alpha$ a non-negative constant.
Then, \eqref{eq:reverse-ii} entails
\begin{align*}
\int\limits_0^\infty \der s \ \e^{-qs} \underline n\left(   h\left( k_{\gamma(s)}\circ \varepsilon\right) \e^{\alpha\gamma(s)} \1_{\{s<V\}} \right) = \int\limits_0^\infty \der s \ \e^{-qs} \underline n\left(   h\circ\rho\left( k_{\gamma(s)}\circ \varepsilon\right) \e^{\alpha\gamma(s)} \1_{\{s<V\}}  \right).
\end{align*}
Under regularity conditions guaranteeing the existence and injectivity of the Laplace transform, this identity implies \eqref{eq:id-n1}.
This is true in particular if both sides in \eqref{eq:id-n1} are right-continuous functions of $s$, for every $s\in(0,+\infty)$, grow at most exponentially and are locally integrable on $[0,+\infty)$ \cite{Zay96}. 
Lemma~\ref{lemma:continuity-on-s} $(i)$ ensures the right-continuity on $(0,+\infty)$, so we will now focus on showing that the r.h.s. in \eqref{eq:id-n1} satisfies the other two conditions. Notice all the arguments we use below also apply when changing $h$ by $h\circ \rho$. 

Let us first show local integrability. For $q>\alpha$, $s>0$ and any constant $K>0$, we have
\begin{align*}
\int_{0}^K \der s \ \underline n\left(   h\left( k_{\gamma(s)}\circ \varepsilon\right) \e^{\alpha\gamma(s)} \1_{\{s<V\}} \right)  
& \leq \Vert h\Vert_\infty \e^{\alpha K} \int_{0}^K \der s \ \underline{n} \left( s<V\right) 
\\
& = \Vert h\Vert_\infty \e^{\alpha K}  \underline{n} \left( V\wedge K\right)
\\
& \leq C_{K,h,\alpha}  \underline{n} \left( V\wedge 1\right),
\end{align*}
where as usual, $\Vert \cdot\Vert_\infty$ stands for the uniform norm, $C_{K,h,\alpha}$ is a positive constant and $\underline{n} \left( V\wedge 1\right)$ is always finite.

The exponential growth condition is also straightforward
\begin{align*}
\underline n\left(   h\left( k_{\gamma(s)}\circ \varepsilon\right) \e^{\alpha\gamma(s)} \1_{\{s<V\}} \right) 
& \leq \Vert h\Vert_\infty \int  \underline{n} \left( \der \varepsilon\right) \e^{\alpha s} \1_{\{s<V\}}
\\
& \leq \Vert h\Vert_\infty  \underline{n} \left( V>s\right) \e^{\alpha s}
\leq C^{\prime\prime} \e^{\alpha s},
\end{align*}
for every $s>1$, since as we have mentioned before, $\underline{n}(V>1)$ is finite.

Thus, we can conclude that both functions in \eqref{eq:id-n1} have well-defined Laplace transforms that satisfy the injectivity conditions, so this identity holds for every $s\geq 0$ (is trivial for $s=0$).

\item In order to prove the second identity we will follow a similar path. Go back to Equation~\eqref{eq:reverse-ii}, on its r.h.s., we disintegrate $\underline n$ with respect to $\varepsilon_{\gamma(s)}$, getting
\begin{align}
&\int\limits_0^\infty \der{s} \ \e^{-qs} \int\limits_{x\in (0,+\infty)}
\underline n\left( f\circ\rho\left( k_{\gamma(s)}\circ \varepsilon\right) \1_{\{\varepsilon_{\gamma(s)}\in \der x,s<V\}} \right) E\left[ \1_{\{x>Y_q\}} g \left(x,x +Z_q\right) \right] \nonumber
\\
&= \int\limits_0^\infty \der{s} \ \e^{-qs} \int\limits_{x\in (0,+\infty)}
\underline n\left( f\left( k_{\gamma(s)}\circ \varepsilon\right) \1_{\{\varepsilon_{\gamma(s)}\in \der x,s<V\}} \right) E\left[ \1_{\{x>Y_q\}} g \left(x,x +Z_q\right) \right] \nonumber
\\
& = \int\limits_0^\infty \der{s} \ \e^{-qs} \underline n\left(  f\left( k_{\gamma(s)}\circ \varepsilon\right) \1_{\{s<V\}} E\left[ \1_{\{\varepsilon_{\gamma(s)}>Y_q\}} g \left(\varepsilon_{\gamma(s)},\varepsilon_{\gamma(s)} +Z_q\right) \right] \right) \nonumber
\\
& =\int\limits_0^{\infty} \der s \ \e^{-qs} \underline n\left(  f\circ\rho\left( k_{\gamma(s)}\circ \varepsilon\right) \1_{\{s<V\}} P\left(\varepsilon_{\gamma(s)}>Y_q \right) g \left(\varepsilon_{\gamma(s)},\varepsilon_s \right) \right). \nonumber
\end{align}
We have applied the identity \eqref{eq:id-n1} for the first step, which implies in particular that for every $f\in\mathcal{H}$ and $s> 0$,
\begin{align*}
\underline n\left(  f\circ\rho\left( k_{\gamma(s)}\circ \varepsilon\right) \1_{\{\varepsilon_{\gamma(s)}\in \der x,s<V\}} \right) 
= \underline n\left( f\left( k_{\gamma(s)}\circ \varepsilon\right) \1_{\{\varepsilon_{\gamma(s)}\in \der x,s<V\}}\right),
\end{align*}
since $\varepsilon_{\gamma(s)}$ is invariant under time reversal of the excursion $\varepsilon\circ k_{\gamma(s)}$ and $f\circ\rho\circ\rho \equiv f$. 
The latter argument also justifies the last equality, when applied directly to \eqref{eq:reverse-ii}.

Moreover, since this is also equal to the l.h.s. in \eqref{eq:reverse-ii} we have
\begin{align}
& \int\limits_0^{\infty} \der s \ \e^{-qs} \underline n\left(  f\left( k_{\gamma(s)}\circ \varepsilon  \right) \1_{\{s<V\}} P\left(\varepsilon_{\gamma(s)}>Y_q \right) g \left(\varepsilon_{\gamma(s)},\varepsilon_s \right) \right) \nonumber
\\
& =\int\limits_0^{\infty} \der s \ \e^{-qs} \underline n\left(  f\circ\rho\left( k_{\gamma(s)}\circ \varepsilon \right) \1_{\{s<V\}} P\left(\varepsilon_{\gamma(s)}>Y_q \right) g \left(\varepsilon_{\gamma(s)},\varepsilon_s \right) \right).
\label{eq:reverse-iii}
\end{align}

We choose the function $g$ as follows, for $0<z<x$,
\begin{align*}
g (x,z) = \dfrac{E_z\left[H\left(k_{T_0}\circ X,x\right)\right]}{P\left(Y_q<x\right)},
\end{align*}
where $H:\mathcal{E} \times (0,+\infty) \rightarrow \R_+$ is a measurable function. 

Then, the Markov property of $\underline{n}$ yields, for every $s\geq 0$,
\begin{align*}
& \underline n\left(  f\left( k_{\gamma(s)}\circ \varepsilon\right) \1_{\{s<V\}} E_{\varepsilon_s}\left[ H\left(k_{T_0}\circ X,\varepsilon_{\gamma(s)}\right) \right] \right) 
\\
& \qquad =  \underline n\left(  f\left( k_{\gamma(s)}\circ \varepsilon\right) \1_{\{s<V\}} H\left(\theta_s\circ\varepsilon,\varepsilon_{\gamma(s)}\right) \right).
\end{align*}
In particular, for $H (\varepsilon, x) = \1_{\{\sup \varepsilon<x\}}$ this is equal to
\begin{align*}
&\underline n\left(  f\left( k_{\gamma(s)}\circ \varepsilon\right) \1_{\{s<V\}} \1_{\{\sup \theta_s\circ \varepsilon<\varepsilon_{\gamma(s)}\}} \right).
\end{align*}
Consequently, for this choice of $g$ and $H$, Equation~\eqref{eq:reverse-iii} becomes,
\begin{align}
& \int\limits_0^{\infty} \der s \ \e^{-qs} \underline n\left(  f\left( k_{\gamma(V)}\circ\varepsilon \right) \1_{\{\gamma(V)<s<V\}}  \right) \nonumber
\\
& \quad =\int\limits_0^{\infty} \der s \ \e^{-qs} \underline n\left(  f\circ\rho\left( k_{\gamma(V)}\circ\varepsilon\right) \1_{\{\gamma(V)<s<V\}}   \right). \label{eq:reverse-iv}
\end{align}

Here again we have the identity between the Laplace transform of two functions of $s$. 
When $f(\varepsilon)=h(\varepsilon)\e^{\alpha V}$ for some $h$ bounded and $\alpha>0$, both integrands are locally integrable on $[0,+\infty)$ by the same arguments used previously for $(i)$.
Right-continuity also holds thanks to Lemma~\ref{lemma:continuity-on-s} $(ii)$, so we can invert the Laplace transform in \eqref{eq:reverse-iv}, leading to the desired identity for $s>0$.  
\end{enumerate}
\end{proof}

\begin{proof}[Proof of Proposition~\ref{prop:pre-sup}]

Let us go back to Equation~\eqref{eq:id-nV}. We can replace $f(\cdot)$ by $h(\cdot) \e^{q\gamma(V(\cdot))}$, for $h$ a bounded mesurable function. Then we integrate both sides with respect to $q\e^{-qs}\der s$ and apply Fubini's theorem, which leads to
\begin{align*}
& \underline n\left( h\left( k_{\gamma(V)}\circ\varepsilon\right) \left( 1 - \e^{-q (V-\gamma(V))} \right) \right)
=\underline n\left( h\circ \rho\left( k_{\gamma(V)}\circ\varepsilon\right) \left( 1 - \e^{-q (V-\gamma(V))} \right) \right),
\end{align*}
or equivalently,
\begin{align*}
& \underline n\left( h\left( k_{\gamma(V)}\circ\varepsilon\right) \left( 1 - \e^{-q (V-\gamma(V))} \right) \1_{\{V-\gamma(V)>0\}} \right) 
\\
& \qquad =\underline n\left( h\circ \rho\left( k_{\gamma(V)}\circ\varepsilon\right) \left( 1 - \e^{-q (V-\gamma(V))} \right) \1_{\{V-\gamma(V)>0\}} \right).
\end{align*}
Since $h$ is non-negative, monotone convergence applies when $q\rightarrow +\infty$, leading to
\begin{align*}
& \underline n\left( h\left( k_{\gamma(V)}\circ\varepsilon\right) \1_{\{V-\gamma(V)>0\}} \right) 
 =\underline n\left( h\circ \rho\left(k_{\gamma(V)}\circ\varepsilon\right) \1_{\{V-\gamma(V)>0\}} \right),
\end{align*}
Finally, notice that $\underline n(\gamma(V)\ge V)=0$, so the identity in the proposition holds for any $h$ bounded. This is still true for any non-negative function, again by a monotone convergence argument.
\end{proof}

\subsection*{Post-supremum process}

We now give a result  analogous to Proposition~\ref{prop:pre-sup} for the post-supremum process of the excursions of $X-I$ away from 0.

\begin{prop}\label{prop:post-sup}
The post-supremum process of the excursion of $X-I$ away from zero is invariant under time reversal, that is, for any measurable functional $h:\mathcal{E}\rightarrow \R_+$,
\begin{align}\label{eq:post-sup}
\underline{n}\left( h\left( \theta^\prime_{\gamma(V)} \circ \varepsilon\right)\right) = \underline{n}\left( h\circ \rho \left( \theta^\prime_{\gamma(V)} \circ \varepsilon\right)\right).
\end{align}
\end{prop}

Before we proceed to prove this result, we need to establish some lemmas, and the following proposition, which is interesting in its own right, since besides serving to prove our main results, it gives the invariance under time-reversal of parts of the trajectory of a killed SPLP.

\begin{prop}\label{prop:rotation-killed-x}
Let $x>0$ and $X$ be a SPLP starting at 0 and killed upon hitting $(-\infty,-x)$. This process shifted to the (unique) value where it attains its supremum before $T_{-x}$ is invariant by rotation.
More precisely, for every $x>0$, $P$-a.s.
\begin{align*}
\theta^\prime_{\gamma(V)} \circ k_{T_{-x}} \circ X  \overset{\text{d}}{=} \rho\left( \theta^\prime_{\gamma(V)} \circ k_{T_{-x}} \circ X \right).
\end{align*}
\end{prop}

We also need the following lemma.
\begin{lemma}\label{lemma:continuity-z}
For every $x\in(0,+\infty)$ and every functional $h\in\mathcal{C}_b(\mathcal{E},\R_+)$, the function $z:[0,+\infty)\rightarrow	[0,+\infty)$ defined as 
\begin{align*}
z(x) \coloneqq E\left[ h\left( \theta^\prime_{\gamma(V)} \circ k_{T_{-x}} \circ X \right) \right]
\end{align*}
is right-continuous.
\end{lemma}
\begin{proof}
See Section~\ref{sec:proofs}.
\end{proof}

\begin{proof}[Proof of Proposition~\ref{prop:rotation-killed-x}]
To demonstrate this result we follow a similar path to that of the proof of Lemma~\ref{lemma:reverse-pre-sup}.
We start by considering the complementary of the event $\A$ defined at the beginning of this section, that is 
\begin{align*}
\A^c (k_t\circ X) = \{\overline{\sigma}_t(X) < \underline{\sigma}_t(X) \}.
\end{align*}
Notice that we may have $\Delta X_{\underline{\sigma}_t}\neq 0$, in particular in the finite variation case in which the excursions of $X-I$ away from 0 starts by a jump \cite{Cha96,ChDo05}. In order to make the notation less heavy we develop the proof only for the infinite variation case. Just a few modifications are needed to treat the general case. 
Likewise, the bounded variation case is a straightforward consequence of Lemma 3.8 in \cite{DaLa15}, where, conditionally on $\varepsilon_\gamma=x$, the post-supremum process is shown to have the law $P_x(\cdot\vert T_0< T_x)\circ k_{T_0}^{-1}$, which is invariant under time-reversal.
\medskip

We define the functional $F_2$ as follows
\begin{align*}
F_2\left( k_t\circ X \right) = f \left(\theta^\prime_{\gamma(V)}\circ k_{\underline{\sigma}_t}\circ X \right) \1_{\A^c(k_t\circ X)} g\left(X_t\right),
\end{align*}
where $f$ and $g$ are non-negative measurable functions. 
In order to apply the duality property \eqref{eq:duality} to this function, let us look at $F_2\circ \rho( k_t\circ X )$ or equivalently $F_2(\rho\circ k_t\circ X)$.
Notice first that $\A^c$ and $X_t$ are invariant under time-reversal at $t$, and additionally, under $\A^c$, we also have $\overline{\sigma}_t = \overline{\sigma}_{\underline{\sigma}_t}$, so it holds that
\begin{align*}
 E\left[ f \left(\theta^\prime_{\gamma(V)}\circ k_{\underline{\sigma}_t}\circ X \right) \1_{\A^c(k_t\circ X)} g\left(X_t\right) \right] 
= E\left[  f\circ \rho \left(\theta^\prime_{\gamma(V)}\circ k_{\underline{\sigma}_t}\circ X \right) \1_{\A^c(k_t\circ X)} g\left(X_t\right)\right].
\end{align*}
Let us integrate this equality in $t$ against the Lebesgue measure,
\begin{subequations}
\label{eq:reverse-killed-x-i}
\begin{align}
\label{eq:reverse-killed-x-i-lhs}
& \int_0^{+\infty} \der t \ E\left[ f \left(\theta^\prime_{\gamma(V)}\circ k_{\underline{\sigma}_t}\circ X \right) \1_{\A^c(k_t\circ X)} g\left(X_t\right) \right] 
\\
\label{eq:reverse-killed-x-i-rhs}
& \qquad \qquad = \int_0^{+\infty} \der t \ E\left[  f\circ \rho \left(\theta^\prime_{\gamma(V)}\circ k_{\underline{\sigma}_t}\circ X \right) \1_{\A^c(k_t\circ X)} g\left(X_t\right)\right].
\end{align}
\end{subequations}
Using the same strategy as before, we can express some quantities in this equation in terms of the excursion straddling $t$ of the process reflected at its infimum.
We recall that $(\tau_u)_{u\geq 0}$ denotes the inverse of the local time at $0$ of the process $X-I$, and $\epsilon_u$ the excursion starting at $\tau_{u-}$. 
Recall also that $-I$ is the local time at 0 for this excursion process and its inverse is $\tau_u = T_{-u} = T_{(-\infty,-u)}$ a.s..
Thus, we can expand \eqref{eq:reverse-killed-x-i-lhs} as follows
\begin{align*}
&\int\limits_0^\infty \der t \ E\left[ f \left(\theta^\prime_{\gamma(V)}\circ k_{\underline{\sigma}_t}\circ X \right) \1_{\A^c(k_t\circ X)} g\left(X_t\right) \right] 
\\
& = \int\limits_0^\infty \der t \ E \left[ \sum\limits_{u:\Delta \tau_u>0} \1_{\{\tau_{u-}<t\leq \tau_u\}} f \left(\theta^\prime_{\gamma(V)}\circ k_{\tau_{u-}}\circ X \right) \1_{\{\sup_{(0,t-\tau_{u-})}\epsilon_u<S_{\tau_{u-}}-I_{\tau_{u-}}\}} g\left(X_t\right) \right].
\end{align*}
Thanks to Fubini's theorem and the change of variable $s=t-\tau_{u-}$, followed by the application of the compensation formula, we obtain that this is equal to
\begin{align*}
& = E \left[  \sum\limits_{u:\Delta \tau_u>0} \int\limits_0^{\Delta \tau_u} \der s \ f \left(\theta^\prime_{\gamma(V)}\circ k_{\tau_{u-}}\circ X \right) \1_{\{\sup_{(0,s)}\epsilon_u<S_{\tau_{u-}}-I_{\tau_{u-}}\}} g\left(X_{\tau_{u-}+s}\right) \right]
\\
& = E \left[  \int_0^\infty \der u  f \left(\theta^\prime_{\gamma(V)}\circ k_{\tau_u}\circ X \right) \int \underline{n} (\der \varepsilon) \int_0^V \der s \  \1_{\{\sup_{(0,s)}\varepsilon<S_{\tau_u}-I_{\tau_u}\}} g\left(X_{\tau_u}+ \varepsilon_s\right) \right]
\\
& = \int_0^\infty \der u\ \e^{-q u} E \left[   f \left(\theta^\prime_{\gamma(V)}\circ k_{T_{-u}}\circ X \right) C_q \left( S_{T_{-u}}-I_{T_{-u}} \right) \right],
\end{align*}
where for any $y\geq 0$ we set ${C_q(y) =\underline{n} \left( \int_0^V \der s\ \e^{q\varepsilon_s}  \1_{\{\varepsilon_{\gamma(s)}<y\}}\right)}$ and we have taken the function $g$ of the form $g(x) = \e^{qx}$, with $q>0$. Fubini's theorem was applied again in the last step.
Note that 
\begin{align*}
C_q(y) =\underline{n} \left( \int_0^V\der s\  \e^{q\varepsilon_s} \1_{\{\bar\varepsilon_{s}<y\}} \right) = \underline{n} \left( \int_0^V\der s\  \e^{q\varepsilon_s} \1_{\{s<T_y\}} \right), 
\end{align*}
so that for any $y>0$,
\begin{align*}
0<\underline{n} \left( T_y\wedge V\right)
\le
C_q(y) \le \e^{qy}  \underline{n} \left(T_y\wedge V\right)<\infty. 
\end{align*}

Now notice that $S_{T_{-u}}-I_{T_{-u}}$ is precisely the height of the excursion ${\theta^\prime_{\gamma(V)}\circ k_{T_{-u}}\circ X}$, which allows us to choose the function $f$ as follows,
\begin{align*}
f(\omega) = \dfrac{h(\omega)}{C_q \left( \sup \omega - \inf \omega\right) },
\end{align*}
for any $\omega\in \mathcal{E}$, where $h$ is a bounded, measurable and positive function. This is a valid choice for $f$, since for all $u>0$, writing $\omega_u =\theta^\prime_{\gamma(V)}\circ k_{T_{-u}}\circ X$, we have that $P(\sup \omega - \inf \omega\in(0,\infty))=1$.

Taking into account the fact that the height of the excursion $\theta^\prime_{\gamma(V)}\circ k_{T_{-u}}\circ X$ is invariant by the transformation $\rho$, Equation~\eqref{eq:reverse-killed-x-i} becomes
\begin{align*}
\int_0^\infty \der u \ \e^{-q u}  E \left[ h\left(\theta^\prime_{\gamma(V)}\circ k_{T_{-u}}\circ X \right) \right]
= \int_0^\infty \der u \ \e^{-q u}  E \left[ h\circ \rho\left(\theta^\prime_{\gamma(V)}\circ k_{T_{-u}}\circ X \right) \right].
\end{align*}
We again have an identity between Laplace transforms of two functions. In virtue of Lemma~\ref{lemma:continuity-z},  if we choose $h\in \mathcal{C}_b(\mathcal{E},\R_+)$, this functions are right-continuous on $(0,+\infty)$. Besides, because $h$ is bounded they are both bounded  (and so locally integrable and growing at most exponentially), hence by the same reasoning as in the proof of Lemma \ref{lemma:reverse-pre-sup}, we can invert the Laplace transforms and write that for any $h\in \mathcal{C}_b(\mathcal{E},\R_+)$ 
\begin{align*}
E \left[ h\left(\theta^\prime_{\gamma(V)}\circ k_{T_{-u}}\circ X \right) \right]
= E \left[ h\circ \rho\left(\theta^\prime_{\gamma(V)}\circ k_{T_{-u}}\circ X \right) \right],
\end{align*}
which completes the proof.
\end{proof}

\begin{proof}[Proof of Proposition~\ref{prop:post-sup}]
By applying the Markov property of $\underline{n}$ at $s>0$, we get 
\begin{align}
& \underline{n}\left( h\left( \theta^\prime_{\gamma(V)} \circ \varepsilon\right) \1\left( s<\gamma(V)\right)\right)
= \int_{x\in(0,+\infty)} \int_{y\geq x} \underline{n}\left( \1\left( s<\gamma(V), \varepsilon_s\in\der x, \varepsilon_{\gamma(V)}\in \der y\right) h\left( \theta^\prime_{\gamma(V)} \circ \varepsilon\right)\right) \nonumber
\\
& \quad = \int_{x\in(0,+\infty)} \int_{y\geq x} \underline{n}\left(\varepsilon_s\in\der x, \overline{\varepsilon}_{s}< y,s<V\right)  E_x\left[ h\left(\theta^\prime_{\gamma(V)}\circ k_{T_{0}}\circ X \right) \1\left(\sup_{[0,T_0]}X\in\der y\right)\right] \nonumber
\\
& \quad = \int_{x\in(0,+\infty)} \int_{y\geq x} \underline{n}\left(\varepsilon_s\in\der x, \overline{\varepsilon}_{s}< y,s<V\right)  E_x\left[ h\circ \rho\left(\theta^\prime_{\gamma(V)}\circ k_{T_{0}}\circ X \right)\1\left(\sup_{[0,T_0]}X\in\der y\right) \right] \nonumber
\\
& \quad = \underline{n}\left( h\circ \rho \left( \theta^\prime_{\gamma(V)} \circ \varepsilon\right)\1\left( s<\gamma(V)\right)\right), \label{eq:post-sup-cvm}
\end{align}
where we have used Proposition~\ref{prop:rotation-killed-x} in the third line, which is possible since $P_x$-a.s.
\begin{align*}
\sup_{[0,T_0]}X = \sup \theta^\prime_{\gamma(V)}\circ k_{T_{0}}\circ X = \sup \rho\left( \theta^\prime_{\gamma(V)}\circ k_{T_{0}}\circ X \right).
\end{align*}
Finally, since $h$ is non-negative, the monotone convergence theorem can be applied to \eqref{eq:post-sup-cvm} letting $s\downarrow 0$, which allows us to conclude.
\end{proof}

We are now ready to prove our main result, stated in the introduction, that we recall now.

\thmain*

\begin{proof}
In the unbounded variation case, $0$ is regular for both half-lines, 
so we can apply Theorem~4.10 from \cite{Duq03}, which ensures that the supremum of the excursion of $X-I$ away from zero, i.e. $\varepsilon_{\gamma(V)}$, admits a density w.r.t. to Lebesgue measure under $\underline{n}$. What is more, this theorem states that for every $x>0$, the pre and post-supremum subpaths are independent under $\underline{n}(\cdot\vert \varepsilon_\gamma=x)$. 
When the trajectories have finite variation, the conditional independence also holds, this result is due to \cite{Mil73,GrPi80} and also \cite{Cha94,Cha96}.
Hence, we can disintegrate by the law of $\varepsilon_\gamma$ and use this independence property, which together with Propositions~\ref{prop:pre-sup} and \ref{prop:post-sup}, lead to the following identities
\begin{align*}
&\underline{n}\left( F\left( \varepsilon \right) \right) =  \underline{n}\left( F\left( \left[ k_\gamma \circ \varepsilon, \theta_\gamma\circ \varepsilon\right] \right) \right)
= \int\limits_{x\in (0,+\infty)} \underline{n}\left( F\left( \left[ k_\gamma \circ \varepsilon, \theta_\gamma\circ \varepsilon\right] \right) \1_{\{\varepsilon_\gamma \in \der x\}} \right)
\\
& = \int\limits_{x\in (0,+\infty)} \underline{n}\left( F\left( \left[ k_\gamma \circ \varepsilon, \theta_\gamma\circ \varepsilon\right] \right) \,\middle|\, \varepsilon_\gamma=x\right) \underline{n}\left( \varepsilon_\gamma \in \der x \right)
\\
& = \int\limits_{x\in (0,+\infty)} \underline{n}\left( F\left( \left[ k_\gamma \circ \varepsilon, \theta^\prime_\gamma\circ \varepsilon + x \right] \right) \,\middle|\, \varepsilon_\gamma=x\right) \underline{n}\left( \varepsilon_\gamma \in \der x \right)
\\
& = \int\limits_{x\in (0,+\infty)} \int F\left( \left[ \eta, \eta' \right] \right) \underline{n}\left( k_\gamma \circ \varepsilon\in \der \eta,\theta^\prime_\gamma\circ \varepsilon + x \in\der \eta' \,\middle|\, \varepsilon_\gamma=x\right) \underline{n}\left( \varepsilon_\gamma \in \der x \right) 
\\
& = \int\limits_{x\in (0,+\infty)} 
\int\int F\left( \left[ \eta, \eta' \right] \right) \underline{n}\left( k_\gamma \circ \varepsilon\in \der \eta \,\middle|\, \varepsilon_\gamma=x\right) \underline{n}\left( \theta^\prime_\gamma\circ \varepsilon + x \in\der \eta' \,\middle|\, \varepsilon_\gamma=x\right) \underline{n}\left( \varepsilon_\gamma \in \der x \right) 
\\
& = \int\limits_{x\in (0,+\infty)} 
\int\int F\left( \left[ \eta, \eta'  \right] \right) \underline{n}\left( \rho\left(k_\gamma \circ \varepsilon\right) \in \der \eta \,\middle|\, \varepsilon_\gamma=x\right) \underline{n}\left( \rho\left(\theta^\prime_\gamma\circ \varepsilon\right) + x \in\der \eta' \,\middle|\, \varepsilon_\gamma=x\right) \underline{n}\left( \varepsilon_\gamma \in \der x \right) 
\\
& = \int\limits_{x\in (0,+\infty)} \underline{n}\left( F\left( \left[ \rho\left(k_\gamma \circ \varepsilon\right), \rho\left(\theta^\prime_\gamma\circ \varepsilon\right) + x \right] \right) \,\middle|\, \varepsilon_\gamma=x\right) \underline{n}\left( \varepsilon_\gamma \in \der x \right)
\\
& = \underline{n}\left( F\left( \left[ \rho\left(k_\gamma \circ \varepsilon\right), \rho\left(\theta^\prime_\gamma\circ \varepsilon\right) +\varepsilon_\gamma \right] \right) \right) = \underline{n}\left( F\left( \chi\circ \varepsilon \right) \right),
\end{align*}
which terminates the proof.
\end{proof}

\cormain*
\begin{proof}
For a fixed path $\varepsilon\in \mathcal{E}$ corresponding to an excursion of $X-I$ away from 0, let us identify the local time process of $\chi(\varepsilon)$, defined by the occupation density formula \eqref{eq:loc-time-inf-var}, as the measurable function $\left(\Gamma(\chi(\varepsilon),r),r\geq 0\right)$ satisfying 
\begin{align}\label{eq:loc-time-chi}
\int\limits_{0}^{V(\varepsilon)}\phi\left([\chi(\varepsilon)]_s\right) \der s = \int\limits_{0}^{\infty} \Gamma\left(\chi(\varepsilon),r \right) \phi(r)	\der r,
\end{align}
for any continuous function $\phi$ with compact support in $[0,\infty)$.
Its existence is guaranteed by Theorem~\ref{th:main} since $\chi(\varepsilon)\overset{\text{d}}{=}\varepsilon$.
The l.h.s. in this equation can be expanded in the following way (writing $\gamma = \gamma(V)$)
\begin{align}
\int\limits_{0}^{V(\varepsilon)}\phi\left([\chi(\varepsilon)]_s\right) \der s 
& = \int\limits_{0}^{\gamma(V)}\phi\left([\chi(\varepsilon)]_s\right) \der s + \int\limits_{\gamma(V)}^V \phi\left([\chi(\varepsilon)]_s\right) \der s \nonumber
\\
& = \int\limits_{0}^{\gamma(V)}\phi\left(\varepsilon_\gamma - \varepsilon_{(\gamma - s)-} \right) \der s + \int\limits_{\gamma(V)}^V \phi\left(\varepsilon_\gamma - \varepsilon_{(\gamma + V - s)-} \right) \der s \nonumber
\\
& = \int\limits_{0}^{V(\varepsilon)}\phi\left(\varepsilon_\gamma - \varepsilon_s\right) \der s
= \int\limits_0^\infty \Gamma \left(\chi'(\varepsilon),r \right)\phi(r) \der r, \label{eq:loc-time-chi-chiprime}
\end{align}
where $[\chi'(\varepsilon)]_s = \varepsilon_\gamma-\varepsilon_s$, for any $s\geq 0$. 
On the other hand, we know from Theorem~\ref{th:main} that for any function $\phi$ satisfying the conditions mentioned before, we have
\begin{align*}
\int\limits_0^\infty \Gamma \left(\chi(\varepsilon),r \right)\phi(r) \der r \overset{\text{d}}{=} \int\limits_0^\infty \Gamma \left(\varepsilon,r \right)\phi(r) \der r.
\end{align*}
Finally, this identity, together with \eqref{eq:loc-time-chi} and \eqref{eq:loc-time-chi-chiprime}, imply that
\begin{align*}
\Gamma\left(\varepsilon,\cdot\right) \overset{\text{d}}{=} \Gamma\left(\chi'(\varepsilon),\cdot\right),
\end{align*}
which terminates the proof.
\end{proof}

\begin{remark}\label{remark:Duq03}
Beyond the independence between the pre and post-supremum subpaths, Theorem~4.10 in \cite{Duq03} gives the following characterization of the law of the pre and post-supremum processes under $\underline{n}(\cdot\vert\varepsilon_\gamma=x)$, for any $x>0$, in terms of the laws $P^\uparrow$ and $P^\downarrow$, corresponding to the process $X$ conditioned respectively to stay positive or negative. We refer to \cite{Ber93,Ber96} and also \cite{Duq03} for the details on the construction of these laws.
The result in the aforementioned theorem has the following implications in our setting in the case of infinite variation. For every positive measurable functional $h:\mathcal{E}\rightarrow \R_+$ we have
\begin{align*}
\underline{n}\left(h\left(k_{\gamma(V)} \circ \varepsilon\right) \,\middle|\, \varepsilon_\gamma=x\right) 
& = 
E^{\uparrow}\left[h\left(k_{T_{x} }\circ X \right)\,\middle|\, X_{T_{x}} = x\right], 
\\
\underline{n}\left(h\left(\theta^\prime_{\gamma(V)} \circ \varepsilon\right) \,\middle|\, \varepsilon_\gamma=x\right) 
& = 
E^{\downarrow}\left[h\left(k_{T_{-x} }\circ X \right)\right].
\end{align*}
Thus, these identities combined with Propositions~\ref{prop:post-sup} and \ref{prop:pre-sup}  imply that for any $x>0$, the laws $P^{\uparrow}\circ k^{-1}_{T_{x}}\left(\cdot\middle\vert X_{T_{x}} = x\right) $ and  $P^{\downarrow}\circ k^{-1}_{T_{-x} }$ are also invariant by rotation.
Additionally, we know from \cite[Chapter~VII]{Ber96} 
that when the process drifts to $-\infty$, the law $P^\downarrow$ can be viewed  as the conditional law $\overline{n}(\cdot\vert V=\infty)$, or equivalently, as the law of $X-S$ shifted at its last passage time at the origin.
Here $\overline{n}$ denotes the excursion measure of $X-S$ away from $0$, defined as in \cite{Duq03} such that it records the final jump of the excursion.
Hence, we also have for any positive  measurable function $h$ that
\begin{align*}
\overline{n}\left( h\left( k_{T_{-x}} \circ \varepsilon\right) \,\middle|\,  V=\infty\right) = \overline{n}\left( h\circ \rho \left( k_{T_{-x}}  \circ \varepsilon\right) \,\middle|\, V=\infty\right).
\end{align*}

\end{remark}

\section{Applications}\label{sec:applications}

The study of the genealogical structure of branching processes is an essential aspect when it comes to their applications in the fields of population dynamics, population genetics and evolutionary biology. In the case of discrete state-space, the genealogy comes naturally from discrete trees, while for continuous-state processes their definition is a more delicate issue and is done via a non-Markovian process called the height process, which was introduced by Le Gall and Le Jan \cite{GaJa98} and is a functional of a SPLP.
We will now briefly outline a few connections between random trees, branching processes and Lévy processes.

\subsection{The continuum random tree}
\label{subsec:CRT}

Real trees can be defined as the continuous limiting object of rescaled discrete trees and can be coded by a continuous function in a way similar to the coding of discrete trees by their contour functions. 
\textit{Aldous' Continuum Random Tree} (the so-called CRT) can be defined as the random real tree coded by a normalized Brownian excursion $\mathbf{e}$, i.e. the positive Brownian excursion conditioned to have lifetime 1. 
More generally, the tree coded by Brownian motion (possibly with drift) reflected at 0, is called Brownian forest.
We refer to \cite{Ald93,Gall05,Eva07} for the formalism on real trees.

\subsubsection{Ray-Knight theorems}

The second Ray-Knight theorem \cite{ReYo91} establishes that the local time process of a reflected Brownian motion is Feller's branching diffusion.
More precisely, let $B$ be a Brownian motion reflected at $0$ and $(L_s^a, s,a\geq 0)$ the family of its occupation densities, where the index $s$ corresponds to the \textit{time} of the original process $B$ and $a$ is the level variable moving in the state-space of $B$. 
Consider, for $x>0$,
\[
\varsigma_x = \inf\{s:L_s^0>x\}.
\]
Then, the process $(L_{\varsigma_x}^t, t\geq 0)$ is equal in distribution to the square of a $0$-dimensional Bessel process started at $x$, that is, a \textit{standard} Feller branching diffusion $(Z_t^x,t\geq 0)$. The latter is defined as the unique strong solution of the SDE
\begin{align*}
\der Z_t^x = 2 \sqrt{Z_t^x} \der W^x_t, \ \textrm{with} \ Z_0^x=x.
\end{align*}

This may be understood as a description of the genealogy encoded in Feller's branching diffusion, meaning that reflected Brownian motion codes (in the sense of Aldous) the \textit{real tree} which describes the genealogy of the population which evolves according to Feller's diffusion \cite{Gall05}.

\subsection{Splitting trees, CMJ's and contour process}
A chronological tree is the subset of $\bigcup_{n\geq 0}\N^n\times [0,+\infty)$ containing all the \textit{existence points} of individuals living for a certain amount of time and giving birth to others during their lifetime. They are represented in the plane, as in Fig.~\ref{fig:contour} (right), with time running from bottom to top, dotted lines representing filiations between individuals: the one on the left is the parent, and that on the right its descendant.
We refer to \cite{Lam10} for the details.

\begin{center}
\begin{figure}
	\includegraphics[scale=.7]{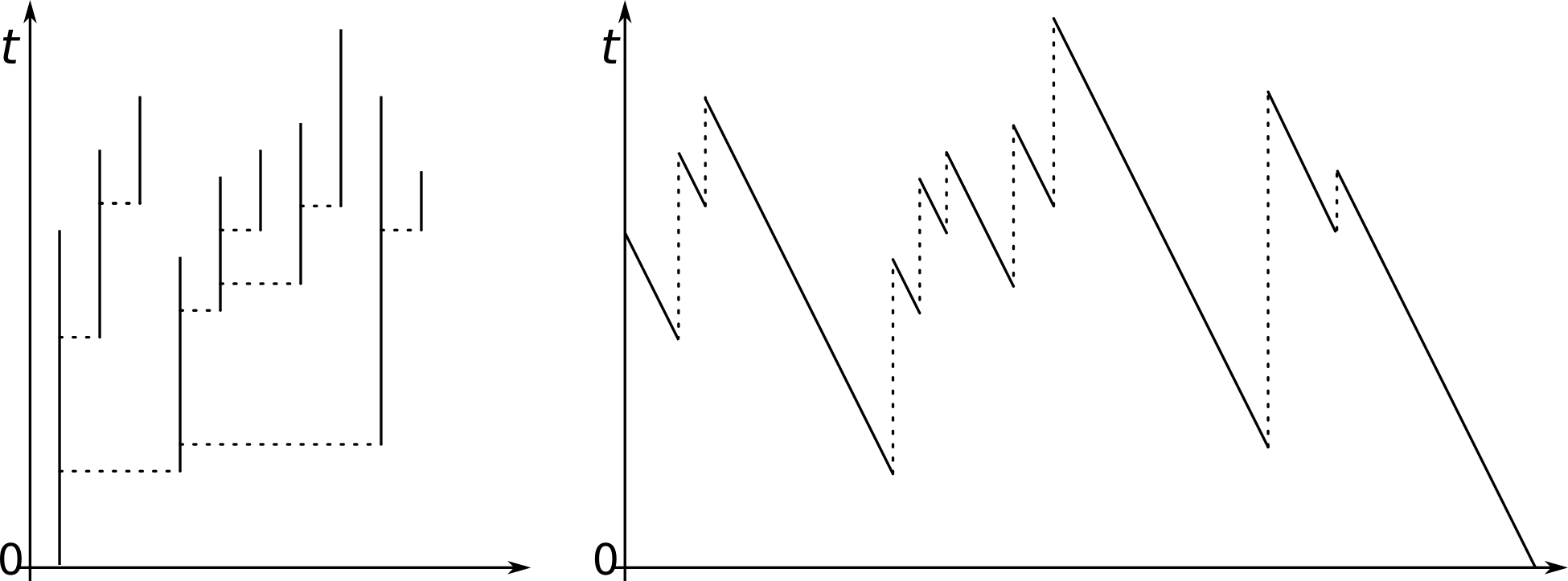}
\caption{An example of chronological tree with finite length (left) and its contour process (right).}
\label{fig:contour}
\end{figure}
\end{center}

Consider a population (or particle system) that originates at time $0$ with one single progenitor, where individuals (particles) evolve independently of each other, giving birth to i.i.d. copies of themselves at constant rate, while alive, and having a lifetime duration with general distribution. 
The family tree under this stochastic model is a \textit{splitting tree}, that can be formally defined as an element $\ctree$ randomly chosen from the set of chronological trees, characterized by a $\sigma$-finite measure $\Pi$ on $(0,\infty]$ called the \textit{lifespan measure}, satisfying $\int_{(0,\infty]}\left(r\wedge 1\right)\Pi(\der r)<\infty$.  
This means that if $\Pi$ has mass $b$, the tree corresponds to a population where individuals have i.i.d. lifetimes distributed as $\Pi(\cdot)\slash b$ and give birth to single descendants throughout their lives at constant rate $b$, all having the same independent behavior. In the general definition individuals may have infinitely many offspring and most of the following results remain valid if $\Pi$ is infinite.
 
We can define the \textit{width} or \textit{population size process} of locally finite chronological trees as a mapping $\Xi$ that maps a chronological tree $\ctree$ to the function $\xi:\R_+\rightarrow \N$ counting the number of extant individuals at time $t\geq 0$
\begin{align*}
\Xi(\ctree) \coloneqq  \left(\xi_t(\ctree), t \geq 0\right).	
\end{align*}
These functions are càdlàg, piecewise constant, from $\R_+$ into $\N$, and are absorbed at 0. 
Then we can define the extinction event $\textrm{Ext} \coloneqq \{\lim_{t\rightarrow\infty} \xi_t\left(\ctree\right)=0\}$ and the time of extinction of the population in a tree as 
\begin{align*}
T_{\textrm{Ext}}\coloneqq \inf\{t\geq 0: \xi_t(\ctree)=0\},
\end{align*}
with the usual convention $\inf \emptyset = \infty$.
A tree, or its width process $\Xi$, is said to be subcritical, critical or supercritical if
\begin{align*}
m\coloneqq\int\limits_{(0,+\infty]}r\Pi(\der r)
\end{align*}
is less than, equal to or greater than 1.

The width process $\Xi(\ctree) =\left(\xi_t(\ctree),t\geq 0\right)$ of a splitting tree is known to be a \textit{binary homogeneous Crump-Mode-Jagers process} (CMJ). This process is not Markovian, unless $\Pi$ is exponential (\textit{birth-death process}) or a Dirac mass at $\{+\infty\}$ (\textit{Yule process}). 

\subsubsection{The contour of a splitting tree}\label{subsec:contour}

As mentioned before, the genealogical structure of a chronological tree can be coded via continuous or càdlàg functions. We focus in particular on the jumping chronological contour process (JCCP) from \cite{Lam10}.
The JCCP of a chronological tree $\ctree$ with finite length $\ell=\ell(\ctree)$ (the sum of lifespans of all individuals), denoted by $\Co(\ctree)$, is a function from $[0,\ell]$ into $\R_+$, that starts at the lifespan of the ancestor  and then runs backward along the right-hand side of this first branch at speed $-1$ until it encounters a birth event, when it jumps up of a height of the lifespan of this new individual, getting to the next tip, and then repeating this procedure until it eventually hits $0$, as we can see in Fig.~\hyperref[contour]{\ref*{fig:contour}} (see \cite{Lam10} for a formal definition).

The JCCP visits all the existence times of each individual exactly once and the number of times it hits a time level, say $s\geq 0$, is equal to the number of individuals in the population at time $s$. 
More precisely, for any finite tree $\ctree$, the local time of its contour process is the population size process, that is
\begin{align*}
\left(\Gamma\left( \Co(\ctree),r\right), 0\leq r\leq T_{\textrm{Ext}} \right) = \Xi(\ctree),
\end{align*}
where $\Gamma$ is defined as in Equation~\eqref{eq:loc-time-fin-var}.

One of the main results in \cite{Lam10} states that the law of $\Co(\ctree)$ when the tree has lifespan measure $\Pi$, conditional on $\textrm{Ext}$ and on the lifespan of the root individual to be $x$,  is a spectrally positive Lévy process $Y$, with Laplace exponent $\psi(\lambda)=\lambda - \int_0^\infty (1-\exp(-\lambda r))\Pi(\der r), \ \lambda\geq 0$, started at $x$, conditioned and killed upon hitting $0$.
A consequence of this result is that, under $P_x$
\begin{align}
\left(\Gamma\left( k_{T_0}\circ Y,r\right), r\geq 0 \right)
\end{align}
is a CMJ with lifespan measure $\Pi$, starting with one progenitor with lifespan $x$.

\medskip
These arguments together with Theorem~\ref{th:main} lead to Corollary~\ref{cor-rev-loc-time-bis}.

\subsection{Other results}

In the same way as we did in Section~\ref{sec:results}, we consider now the excursion process of $X-S$ away from $0$, the canonical excursion is again denoted by $\varepsilon$, and $\overline{n}$ is the excursion measure of this process. As we pointed out in Remark~\ref{remark:Duq03}, this measure is defined here as in \cite{Duq03}, such that it records	 the final jump of the excursion.
Define for any $s\in\R_+$, the first instant at which the excursion attains its minimum on the interval $[0,s]$, that is
\begin{align*}
\nu(s) = \nu(s,\varepsilon) \coloneqq \arginf_{[0,s]} \varepsilon = \inf \left\{s^\prime\in[0,s]: \varepsilon(s^\prime -)=\underline \varepsilon_s\right\}, \nonumber
\end{align*}
where $\underline \varepsilon_{s} \coloneqq \inf_{[0,s]} \varepsilon$.
We write $\nu=\nu(V)$ for the infimum up to the lifetime of the excursion.
Then, the following results can be derived, which are analogous to those obtained in Section~\ref{sec:results}.
\begin{lemma}\label{lemma:others1}
The post-supremum of the excursion of $X-I$ conditioned to have height $x$ has the same distribution as the pre-infimum of the excursion of $X-S$ conditioned to have depth greater than $x$ and killed upon hitting $-x$ for the first time. 
More precisely, for any functional $F\in\mathcal{C}_b(\mathcal{E},\R_+)$, define the conditional expectation $\kappa\left(x\right)\coloneqq \underline{n} \left(F\left(\theta'_{\gamma}\circ \varepsilon\right) \,\middle|\, \varepsilon_{\gamma}=x\right)$. This function has a continuous version which satisfies for all $x>0$
\begin{align}
\kappa \left( x\right) =
\overline{n} \left(F\left(k_{T_{-x}}\circ \varepsilon\right) \,\middle|\, T_{-x}<\infty\right).
\end{align}
\end{lemma}

In order to demonstrate the previous result we need the following lemma on the continuity of the functions in the stated identity. This next lemma is very close to Lemma~\ref{lemma:continuity-z} and the details of the proof are given in Section~\ref{sec:proofs}.
\begin{lemma}\label{lemma:continuity-w}
For every $x>0$ and every functional $F\in\mathcal{C}_b(\mathcal{E},\R_+)$, the function $w:[0,+\infty)\rightarrow	[0,+\infty)$ defined as 
\begin{align*}
w(x) \coloneqq \overline{n}\left( F\left( k_{T_{-x}} \circ \varepsilon \right)\,\middle|\, T_{-x}<+\infty \right)
\end{align*}
is continuous.
\end{lemma}
\begin{proof}
See Section~\ref{sec:proofs}.
\end{proof}

\begin{proof}[Proof of Lemma~\ref{lemma:others1}]
Fix $x>0$ and let us consider the first excursion of $X-I$ with height greater than $x$, denoted $\varepsilon^{(x)}$. Let $G_x$ be the left-end point of this excursion and define
\begin{align*}
A_x & \coloneqq \inf \left\{t>0: X_t-I_t>x \right\} \\
D_x & \coloneqq \inf \left\{t>A_x: I_t=X_t \right\}.
\end{align*}
It is not hard to see that $\varepsilon^{(x)}$ straddles $A_x$ and that $D_x$ is its right-end point.

For any positive measurable function $F$, and any $y,a,$ such that $y>x-a$, we have the following identity for the post-supremum of the excursion $\varepsilon^{(x)}$,
\begin{align}
&E\left[ F(\theta'_\gamma\circ \varepsilon^{(x)})\1\left(\varepsilon^{(x)}(\gamma)\in \der y+ a, -I_{G_x}\in \der a, S_{G_x}<y, A_x<+\infty\right) \right] \nonumber \\
& \ = E\left[ \sum\limits_{u:\Delta \tau_u>0} F(e_u\circ \theta'_\gamma)\1\left(e_u(\gamma)\in \der y +a,\ -I_{\tau_{u-}}\in \der a,\  S_{\tau_{u-}}<y,\ \sup\limits_{s<u} e_s <x\right) \right], \nonumber 
\end{align}
we can apply the compensation formula to obtain this is equal to
\begin{align}
& \ = E\left[ \int\limits_0^\infty \der u \ \1\left(-I_{\tau_{u}}\in \der a,\ S_{\tau_{u}}<y,\ \sup\limits_{s<\tau_u} (X-I)_s <x\right) \right] \underline{n}\left( F\left(\theta'_\gamma\circ  \varepsilon\right),\ \sup \varepsilon_\gamma \in \der y +a\right) \nonumber \\
& \ = \der a \ P\left( S_{T_{-a}}<y,\ \sup\limits_{s<\tau_u} (X-I)_s <x\right)\underline{n}\left( F\left(\theta'_\gamma\circ  \varepsilon\right),\ \varepsilon_\gamma \in \der y +a\right) \nonumber \\
& \ = \der a \ P\left( S_{T_{-a}}<y,\ \sup\limits_{s<\tau_u} (X-I)_s <x\right) \underline{n}\left( \varepsilon_\gamma \in \der y +a\right) \underline{n}\left( F\left(\theta'_\gamma\circ  \varepsilon\right) \,\middle|\,\varepsilon_\gamma = y +a \right), \label{eq:lhs-F1}
\end{align}
where the second line holds since the inverse local time $\tau_u$ is actually $T_{-u}$, hence $-I_{\tau_u}= u$. 
Now notice that $A_x$ is a stopping time for the process $X$, so thanks to the strong Markov property, \eqref{eq:lhs-F1} can also be expressed as follows
\begin{align}
&E\left[ \int\limits_{z\in [x-a,y]} \1\left(A_x\in \der z, -I_{G_x}\in \der a, S_{G_x}<y, A_x<+\infty\right) \right. \nonumber \\
& \qquad \qquad \qquad \qquad \qquad \left. \times E_z\left[ F\left( \theta'_{\gamma(V)}\circ k_{T_{-a}}\circ X\right) \1\left(S_{T_{-a}}\in \der y, T_{-a}<+\infty\right) \right]  \vphantom{\int\limits_{z\in [x-a,y]}}\right] \nonumber \\
& = \int\limits_{z\in [x-a,y]} P\left(A_x\in \der z, -I_{G_x}\in \der a, S_{G_x}<y, A_x<+\infty\right) \nonumber \\
&\ \times P_z\left( S_{T_{-a}}\in \der y, T_{-a}<+\infty\right) E\left[ F\left( \theta'_{\gamma(V)}\circ k_{T_{-a}}\circ X\right) \,\middle|\, S_{T_{-a}}= y, T_{-a}<+\infty \right] \label{eq:rhs-F1}
\end{align}

Additionally, if we let $(\vartheta_u)_{u\geq 0}$ be the inverse of the local time at $0$ of the process $S-X$, we can expand the last factor in the r.h.s. of \eqref{eq:rhs-F1} as follows, with the help of the compensation formula,
\begin{align}
& E_z\left[ F\left( \theta'_{\gamma(V)}\circ k_{T_{-a}}\circ X\right) \1\left(S_{T_{-a}}\in \der y, T_{-a}<+\infty\right) \right] \nonumber \\
& \qquad = E_z\left[ \sum\limits_{u:\Delta \vartheta_u>0} 1\left( \inf (e_u)>y+a, I_{\vartheta_{u-}}>-a, X_{\vartheta_{u-}}\in \der y\right) F\left( k_{T_{-y-a}}\circ e_u \right)\right] \nonumber \\
& \qquad = \int_0^\infty \der u P_z\left( I_{\vartheta_u}>-a, X_{\vartheta_u}\in \der y\right) \overline{n} \left(F\left( k_{T_{-y-a}}\circ \varepsilon \right)\1_{\{T_{-y-a} <+\infty\}} \right). \label{eq:rhs-F1-last-factor}
\end{align}

By choosing $F \equiv 1$ in \eqref{eq:lhs-F1}, \eqref{eq:rhs-F1} and \eqref{eq:rhs-F1-last-factor}, we get the following identities defining a measure on $\R^2$, that depends on $x$ and that we denote by $\mu_x(\der a, \der y)$, i.e. for any $y>x-a$ we have
\begin{align}
&\mu_x(\der a, \der y) \coloneqq \der a \ P\left( S_{T_{-a}}<y,\ \sup\limits_{s<\tau_u} (X-I)_s <x\right) \underline{n}\left( \varepsilon_\gamma \in \der y +a\right) \nonumber \\
& = \int\limits_{z\in [x-a,y]} P\left(A_x\in \der z, -I_{G_x}\in \der a, S_{G_x}<y, G_x<+\infty\right) P_z\left( S_{T_{-a}}\in \der y, T_{-a}<+\infty\right) \nonumber \\
& =\int\limits_{z\in [x-a,y]} P\left(A_x\in \der z, -I_{G_x}\in \der a, S_{G_x}<y, G_x<+\infty\right) \times \nonumber \\
& \qquad \qquad \qquad \qquad \qquad \quad \int_0^\infty \der u P_z\left( I_{\vartheta_u}>-a, X_{\vartheta_u}\in \der y\right) \overline{n} \left(T_{-y-a} <+\infty \right).
 \label{eq:mu_x-dady}
\end{align}

Finally, combining \eqref{eq:lhs-F1}, \eqref{eq:rhs-F1}, \eqref{eq:rhs-F1-last-factor} and \eqref{eq:mu_x-dady} we get that,
\begin{align*}
\mu_x\left(\der a, \der y\right) \underline{n}\left( F\left(\theta'_\gamma\circ  \varepsilon\right) \,\middle|\,\varepsilon_\gamma = y +a \right)
= \mu_x\left(\der a, \der y\right)  \overline{n}\left( F\left( k_{T_{-y-a}}\circ \varepsilon\right) \,\middle|\, T_{-y-a}<+\infty\right).
\end{align*}

Let $\Lambda(\der a, \der y)$ denote the Lebesgue measure in $\R^2$.
We know from \cite[Th. 4.10]{Duq03} that the law of $\varepsilon_\gamma$ is absolutely continuous w.r.t. Lebesgue measure under $\underline{n}$, so $\mu$ has a density with respect to $\Lambda(\der a, \der y)$. Since the previous identity holds $\Lambda$ a.e., and thanks to Lemma~\ref{lemma:continuity-w}, we have that $\overline{n}( F( k_{T_{-x'}}\circ \varepsilon ) \,|\, T_{-x'}<+\infty )$ is continuous for every $x'>0$ and $F\in\mathcal{C}_b(\mathcal{E},\R_+)$, we can conclude that the conditional measure $\underline{n}\left( F\left(\theta'_\gamma\circ  \varepsilon\right) \,\middle|\,\varepsilon_\gamma = x' \right)$ admits a continuous version that is equal to 
\[\overline{n}\left( F\left( k_{T_{-x'}}\circ \varepsilon\right) \,\middle|\, T_{-x'}<+\infty\right),\]
which proves the lemma.
\end{proof}

\begin{prop}
For any $x>0$, the law of $k_{T_{-x}}\circ \varepsilon$ is invariant by space-time reversal under $\overline{n} \left(\cdot \,\middle|\, T_{-x}<\infty\right)$, that is for any measurable $F:\mathcal{E}\to\R_+$,
$$
\overline{n} \left(F\left(k_{T_{-x}}\circ \varepsilon\right) \,\middle|\, T_{-x}<\infty\right) =\overline{n} \left(F\circ\rho\left( k_{T_{-x}}\circ \varepsilon\right) \,\middle|\, T_{-x}<\infty\right).
$$
In addition, 
\begin{equation}
\label{eqn:densidad}
\overline{n}\left(F\left(k_{\nu}\circ \varepsilon\right)\,\middle|\,  -\varepsilon_\nu=x,V<\infty \right)=\overline{n}(F(k_{T_{-x}}\circ \varepsilon) \,|\, T_{-x}<\infty),
\end{equation}
which implies that $k_{\nu}\circ \varepsilon$ is invariant by space-time reversal under $\overline{n}\left(\cdot\,|\,  V<\infty \right)$.
\end{prop}

\begin{proof}
As a consequence of Proposition~\ref{prop:post-sup} and Lemma~\ref{lemma:others1}, the first equation of the Proposition holds for any non-negative $F$ which is continuous and bounded. The result can be extended to any non-negative, bounded $F$ by density and to any non-negative $F$ by monotone convergence.\\
Now recall that the scale function $W$ is almost everywhere differentiable (see \cite{Kyp06,Don07}). Let us show that for every functional $F$ as in the statement, the measure
\begin{align*}
\overline{n}\left(F\left(k_{\nu}\circ \varepsilon\right), -\varepsilon_\nu\in dx\,|\, V<\infty \right),\quad x>0,
\end{align*}
has a density equal to
\begin{align*}
\overline{n}(F(k_{T_{-x}}\circ \varepsilon) \,|\, T_{-x}<\infty) \ \frac{\overline{n}(T_{-x}<\infty \,|\,V<\infty)}{P_{-x}(T_0<\infty)}\ \frac{W'(x)}{W(x)} .
\end{align*}

Applying the strong Markov property of $\underline n(\cdot\,|\, V<\infty)$ at $T_{-x}$, we get that 
\begin{align*}
\overline{n}\left(F\left(k_{T_{-x}}\circ \varepsilon\right), -\varepsilon_\nu\in (x,x+h)\,|\, V<\infty \right)
\end{align*}
is equal to
\begin{eqnarray*}
 &=&\overline{n}\left(F\left(k_{T_{-x}}\circ \varepsilon\right), T_{-x}<\infty \,|\, V<\infty \right)  P_{-x}(T_0<T_{-(x+h)}\,|\, T_0<\infty)\\
  &=&\overline{n}\left(F\left(k_{T_{-x}}\circ \varepsilon\right)\,|\,T_{-x}<\infty)\,\overline{n}(T_{-x}<\infty \,|\, V<\infty \right)  P_{-x}(T_0<T_{-(x+h)}\,|\, T_0<\infty)\\
  &=&\overline{n}(F(k_{T_{-x}}\circ \varepsilon) \,|\, T_{-x}<\infty) \ \frac{\overline{n}(T_{-x}<\infty \,|\,V<\infty)}{P_{-x}(T_0<\infty)}\ P_{-x}(T_0<T_{-(x+h)})\\
  &=&\overline{n}(F(k_{T_{-x}}\circ \varepsilon) \,|\, T_{-x}<\infty) \ \frac{\overline{n}(T_{-x}<\infty \,|\,V<\infty)}{P_{-x}(T_0<\infty)}\ \left(1-\frac{W(x)}{W(x+h)}\right),
\end{eqnarray*}
and since $W$ is differentiable at $x$, we find that 
\begin{multline*}
\lim_{h\downarrow 0} \frac 1h\overline{n}\left(F\left(k_{T_{-x}}\circ \varepsilon\right), -\varepsilon_\nu\in (x,x+h)\,|\, V<\infty \right)
\\= \overline{n}(F(k_{T_{-x}}\circ \varepsilon) \,|\, T_{-x}<\infty) \ \frac{\overline{n}(T_{-x}<\infty \,|\,V<\infty)}{P_{-x}(T_0<\infty)}\ \frac{W'(x)}{W(x)},
\end{multline*}
so we get 
\begin{multline*}
 \overline{n}\left(F\left(k_{T_{-x}}\circ \varepsilon\right), -\varepsilon_\nu\in dx\,|\, V<\infty \right)
 =\overline{n}\left(F\left(k_{\nu}\circ \varepsilon\right), -\varepsilon_\nu\in dx\,|\, V<\infty \right)
\\= \overline{n}(F(k_{T_{-x}}\circ \varepsilon) \,|\, T_{-x}<\infty) \ \frac{\overline{n}(T_{-x}<\infty \,|\,V<\infty)}{P_{-x}(T_0<\infty)}\ \frac{W'(x)}{W(x)}\, dx,
\end{multline*}
which is the announced result.
Now taking $F\equiv 1$, we get 
\begin{align*}
\overline{n}\left(-\varepsilon_\nu\in dx\,|\, V<\infty \right)
 = \frac{\overline{n}(T_{-x}<\infty \,|\,V<\infty)}{P_{-x}(T_0<\infty)}\ \frac{W'(x)}{W(x)}\, dx,
\end{align*}
so combining the last two equalities, we arrive at 
\begin{multline}
 \overline{n}\left(F\left(k_{\nu}\circ \varepsilon\right), -\varepsilon_\nu\in dx\,|\, V<\infty \right)
\\= \overline{n}(F(k_{T_{-x}}\circ \varepsilon) \,|\, T_{-x}<\infty) \ \overline{n}\left(-\varepsilon_\nu\in dx\,|\, V<\infty \right),
\label{eqn:medidas}
\end{multline}
which can also be expressed as in \eqref{eqn:densidad}
\begin{align*}
 \overline{n}\left(F\left(k_{\nu}\circ \varepsilon\right)\,|\, -\varepsilon_\nu=x, V<\infty \right)
= \overline{n}(F(k_{T_{-x}}\circ \varepsilon) \,|\, T_{-x}<\infty).
\end{align*}
Now from \eqref{eqn:medidas} and the first result of the Proposition, we get 
\begin{eqnarray*}
 \overline{n}\left(F\circ\rho\left(k_{\nu}\circ \varepsilon\right), -\varepsilon_\nu\in dx\,|\, V<\infty \right)
&=& \overline{n}(F\circ \rho(k_{T_{-x}}\circ \varepsilon) \,|\, T_{-x}<\infty) \ \overline{n}\left(-\varepsilon_\nu\in dx\,|\, V<\infty \right)\\
&=&\overline{n}(F(k_{T_{-x}}\circ \varepsilon) \,|\, T_{-x}<\infty) \ \overline{n}\left(-\varepsilon_\nu\in dx\,|\, V<\infty \right)\\
&=& \overline{n}\left(F\left(k_{\nu}\circ \varepsilon\right), -\varepsilon_\nu\in dx\,|\, V<\infty \right),
\end{eqnarray*}
and integrating over $x$ the last equality, we get
\begin{align*}
 \overline{n}\left(F\circ\rho\left(k_{\nu}\circ \varepsilon\right)\,|\, V<\infty \right)
=\overline{n}\left(F\left(k_{\nu}\circ \varepsilon\right)\,|\, V<\infty \right),
\end{align*}
which terminates the proof.
\end{proof}

\section{Remaining proofs}\label{sec:proofs}

\begin{proof}[Proof of Lemma~\ref{lemma:continuity-on-s}]
Every function $f\in\mathcal{H}$ can be expressed as $f(\varepsilon) = h(\varepsilon) \e^{\alpha V(\varepsilon)}$, for a non-negative bounded function $h$ and a non-negative constant $\alpha$.
Hence, here we want to prove that for every non-negative bounded function $h$ and any non-negative constant $\alpha$, the functions
\begin{align*}
&\underline n\left(   h\left( k_{\gamma(s)}\circ \varepsilon\right) \e^{\alpha \gamma(s)} \1_{\{s<V\}}\right)
\\
&\underline n\left(   h\left( k_{\gamma(V)}\circ \varepsilon\right) \e^{\alpha \gamma(s)} \1_{\{\gamma(V)<s<V\}}\right)
\end{align*}
are right-continuous at every $s>0$. 

Let us start by $(i)$.
Fix $s>0$, and a sequence $(s_n)\subset\R_+$ such that $s_n\downarrow s$. 
For $\delta>0$, define the following subsets of $\mathcal{E}$:
\begin{align*}
\Upsilon_s(\delta)\coloneqq \left\{\varepsilon\in \mathcal{E} : \overline{\varepsilon}(s-\delta)=\overline{\varepsilon}(s+\delta)\right\},
\end{align*}
Then, we can analyze the continuity of $\underline n\left(   h\left( k_{\gamma(s)}\circ \varepsilon\right)\e^{\alpha \gamma(s)} \1_{\{s<V\}} \right)$ at $s$ by splitting the space $\mathcal{E}$ as follows for any $\delta'>0$
\begin{align*}
& \left| \underline n\left(   h\left( k_{\gamma(s_n)}\circ \varepsilon \right) \e^{\alpha \gamma(s_n)} \1_{\{s_n<V\}}  \right) - \underline n\left(   h\left( k_{\gamma(s)}\circ \varepsilon\right) \e^{\alpha \gamma(s)} \1_{\{s<V\}}  \right) \right| 
\\
&\quad \leq  \int \left| h\left( k_{\gamma(s_n)}\circ \varepsilon\right) \e^{\alpha \gamma(s_n)} \1_{\{s_n<V\}} - h\left( k_{\gamma(s)}\circ \varepsilon\right) \e^{\alpha \gamma(s)} \1_{\{s<V\}} \right| \ \underline n\left(\der \varepsilon \right) 
\\
&\quad = \underbrace{\int\limits_{V\leq s+\delta^\prime} \vert \cdot \vert \ \underline n\left(\der \varepsilon \right) }_{(1)}
+ \underbrace{\int\limits_{\left(\Upsilon_s(\delta)\right)^c,V>s+\delta^\prime} \vert \cdot \vert \ \underline n\left(\der \varepsilon \right) }_{(2)}
+ \underbrace{\int\limits_{\Upsilon_s(\delta),V>s+\delta^\prime} \vert \cdot \vert \ \underline n\left(\der \varepsilon \right) }_{(3)}.
\end{align*}
Now let us see what happens with each of the terms in this sum:
\begin{itemize}
\item[(1)] Since $s_n\geq s$, $0\leq \gamma(s)\leq s$ and $h$ is bounded, we have
\begin{align*}
&\int\limits_{V\leq s+\delta^\prime} \left| h \left( k_{\gamma(s_n)}\circ \varepsilon\right) \e^{\alpha \gamma(s_n)} \1_{\{s_n<V\}}  - h\left( k_{\gamma(s)}\circ \varepsilon\right) \e^{\alpha \gamma(s)} \1_{\{s<V\}} \right| \ \underline n\left(\der \varepsilon \right) 
\\
& \qquad \leq 2 \Vert h\Vert_\infty \e^{\alpha (s+\delta^\prime)}\underline n\left(s<V\leq s+\delta^\prime \right). \nonumber
\end{align*}
For every $s,\delta^\prime>0$ it holds that $\underline n(s<V\leq s+\delta^\prime)<+\infty$. Therefore, downward monotone convergence applies and it implies that 
\[
\underline n\left(s<V\leq s+\delta^\prime\right)\longrightarrow 0, \ \textrm{when} \ \delta^\prime\rightarrow 0.
\]
This allows to choose, for every $\eta>0$, a suitable $\delta^\prime$ such that the term $(1)$ is smaller than $\frac{\eta}{2}$.
\item[(2)] Again, $h$ bounded implies that whenever $s_n<s+\delta'$
\begin{align*}
&\int\limits_{\substack{\left(\Upsilon_s(\delta)\right)^c \\ V>s+\delta^\prime}}  \left| h \left( k_{\gamma(s_n)}\circ \varepsilon\right) \e^{\alpha \gamma(s_n)} \1_{\{s_n<V\}}  - h\left( k_{\gamma(s)}\circ \varepsilon\right) \e^{\alpha \gamma(s)}  \1_{\{s<V\}} \right| \ \underline n\left(\der \varepsilon \right)
\\
& \qquad \qquad \leq 2 \Vert h\Vert_\infty \e^{\alpha (s+\delta^\prime)} \underline n\left(\left(\Upsilon_s(\delta)\right)^c,V> s+\delta^\prime\right).
\end{align*}
On the other hand, from the definition of $\Upsilon_s(\delta)$ and since the supremum is attained at a unique point $\underline{n}$-a.e., it follows from the dominated convergence theorem that

\begin{eqnarray}
\lim\limits_{\delta\rightarrow 0}\underline n\left(\left(\Upsilon_s(\delta)\right)^c,V> s+\delta^\prime \right) =\underline n\left(\varepsilon_s=\overline\varepsilon_s,V> s+\delta^\prime \right). \label{eq:limit-eta}
\end{eqnarray}
We now show that the r.h.s. of this limit is $0$ for every fixed $s>0$. For any $u\in(0, s)$,
\begin{eqnarray}
&&\underline n\left(\varepsilon_s=\overline\varepsilon_s,V> s+\delta^\prime \right) = \int\limits_{x\in(0,+\infty)} \int\limits_{y\geq x}  \underline n \left(\varepsilon_s=\overline\varepsilon_s,V> s+\delta^\prime,\varepsilon_u\in\der x,\overline\varepsilon_u\in\der y \right) 
\nonumber \\
&& \qquad = \int\limits_{x\in(0,+\infty)} \int\limits_{y\geq x}  \underline n \left(\varepsilon_u\in\der x,\overline\varepsilon_u\in\der y \right) P_x\left(T_{0}^-> s+\delta^\prime-u,S_{s-u}=X_{s-u}\geq y\right),
\nonumber 
\end{eqnarray}
where the last line comes from the Markov property. Besides, for $y\geq x$
\begin{eqnarray}
&&P_x\left(T_0^->s+\delta^\prime-u,S_{s-u}=X_{s-u}\geq y\right) = P_0\left(T_{-x}^->s+\delta^\prime-u,S_{s-u}=X_{s-u}\geq y-x\right)
\nonumber \\
&& \qquad \leq P_0\left(S_{s-u}=X_{s-u}\right) = P\left(\exists t>0:\mathscr{L}^{-1}(t) = s-u\right) = P\left(\mathscr{L}^{-1}(\mathscr{T}_{s-u}) = s-u\right),
\nonumber
\end{eqnarray}
where $\mathscr{L}^{-1}$ is the so-called \textit{ladder time process}, which is the inverse of the local time at $0$ of the process reflected at its supremum, $S-X$; and $\mathscr{T}_v\coloneqq \inf\{t:\mathscr{L}^{-1}(t)> v\}$ for any $v\geq 0$. We know from \cite{Ber96} that $\mathscr{L}^{-1}$ is a subordinator, with drift equal to $0$ when $0$ is regular for $(-\infty,0)$, which is always the case in absence of negative jumps (see for instance \cite{Cha13}).
Another result from \cite[Chapter~III.2]{Ber96} tells us that any subordinator $Y$ with drift $0$ never creeps over any level $x>0$, that is $P\left(Y_{T_{x}^+}=x\right) =0$. Hence, we can conclude that
\begin{align}
\underline n\left(\varepsilon_s=\overline\varepsilon_s,V>s+\delta^\prime \right) = 0, \label{eq:e_s=sup_e_s}
\end{align}
which guarantees, together with \eqref{eq:limit-eta}, that for any $\eta>0$, we can choose $\delta<\delta^\prime$ sufficiently small that 
\begin{align}
&\int\limits_{\substack{\left(\Upsilon_s(\delta)\right)^c \\ V>s+\delta^\prime}} \left| h\left( k_{\gamma(s_n)}\circ \varepsilon\right) \e^{\alpha \gamma(s_n)} \1_{\{s_n<V\}} - h\left( k_{\gamma(s)}\circ \varepsilon\right) \e^{\alpha \gamma(s)} \1_{\{s<V\}}  \right| \ \underline n\left(\der \varepsilon \right)  
< \frac{\eta}{2}. \nonumber
\end{align}
\item[(3)] Let $N_\delta$ be such that for $n\geq N_\delta$, $\vert s_n-s\vert<\delta$, then $\forall n\geq N_\delta, \forall \varepsilon\in \Upsilon_s(\delta)$, such that $V(\varepsilon)>s+\delta^\prime$, we have $\gamma(s_n,\varepsilon)=\gamma(s,\varepsilon)$ and $\1_{\{s_n<V\}}=\1_{\{s<V\}}=1$. Hence the third term is 0 for $n\geq N_\delta$.
\end{itemize}
Finally, we can conclude that the function $\underline n\left(   f\left( k_{\gamma(s)}\circ \varepsilon\right) \1_{\{s<V\}}\right)$ is right-continuous for every $s>0$. 

\medskip
For $(ii)$ take as well $s>0$ and $s_n\downarrow s$. Fix $\delta>0$, then there exists $N_\delta$ such that for every $n\geq N_\delta$, $|s_n-s|<\delta$ and also 
\begin{align*}
& \left| \underline n\left(  f\left( k_{\gamma(V)}\circ\varepsilon\right) \1_{\{\gamma(V)<s_n<V\}}  - \underline n\left(  f\left( k_{\gamma(V)}\circ\varepsilon\right) \1_{\{\gamma(V)<s<V\}}  \right) \right)  \right|  
\\
& \qquad \le \int f\left(k_{\gamma(V)}\circ\varepsilon\right)\left| \1_{\{s\leq \gamma(V)<s_n<V\}} - \1_{\{\gamma(V)<s<V\leq s_n\}} \right| \underline{n}(\der \varepsilon)
\\
& \qquad \leq \Vert h\Vert_\infty \e^{\alpha(s+\delta)}\left(\underline{n}\left(s\leq \gamma(V)<s_n<V\right) + \underline{n}\left(\gamma(V)<s<V\leq s_n\right)  \right).
\end{align*}
By dominated convergence as $n\rightarrow\infty$, 
\begin{align*}
\underline{n}\left(\gamma(V)<s<V\leq s_n\right)\longrightarrow \underline{n}(\emptyset)=0
\end{align*}
and 
\begin{align*}
\underline{n}\left(s\leq \gamma(V)<s_n<V\right)\longrightarrow \underline{n}\left(s = \gamma (V), s<V\right) \leq \underline{n}\left( \varepsilon_s=\overline{\varepsilon}_s, s<V\right) = 0,
\end{align*}
as we has just proved in \eqref{eq:e_s=sup_e_s}. So the function in $(ii)$ is also right-continuous.

\end{proof}

\begin{proof}[Proof of Lemma~\ref{lemma:continuity-z}]

Notice first that for all $x>0$, since $X$ has no negative jumps, $X$ is a.s. continuous at $T_{-x}$ and $P\left(\Delta X_{T_{-x}} = 0, T_{-x}=T_{(-\infty,-x)}\right)=1$.
This allow us to apply \cite[Proposition VI.2.11 and VI.2.12]{JaShi03}, which ensure in this context that if we have $x_n\downarrow x$, then a.s. $T_{-x_n}\downarrow T_{-x}$, and moreover, the killed paths $k_{T_{-x_n}}\circ X$ also converge to $k_{T_{-x}}\circ X$ when $n\rightarrow \infty$ in Skorokhod topology. Hence, it exists a sequence $(\lambda_n)$ of changes of time (see Section~\ref{sec:preliminaries}) such that $\Vert \lambda_n - \textrm{Id}\Vert_\infty \rightarrow 0$ and $\Vert k_{T_{-x_n}}\circ X\circ \lambda_n -k_{T_{-x}}\circ X\Vert_M \rightarrow 0$ for all $M\geq 0$ a.s.

Additionally, since the sequence $(x_n)$ is decreasing, we deduce from the definition of $\gamma$ that $\left(\gamma(T_{-x_n},X)\right)$ is also a decreasing sequence, and that for all $n\geq 0$ we have
\begin{align*}
\gamma(T_{-x_n},X) \geq \gamma(T_{-x},X).
\end{align*}
Hence $\gamma(T_{-x_n},X)\downarrow \ell$ for some $\ell\geq 0$. 
Suppose that $\ell > \gamma(T_{-x},X)$, this implies that for every $n\geq 0$, $T_{-x}<\gamma(T_{-x_n},X)$, so we have
\begin{align*}
T_{-x}<\gamma(T_{-x_n},X) \leq T_{-x_n}.
\end{align*}
Then, the convergence of $(T_{-x_n})$ entail that $\gamma(T_{-x_n},X)\downarrow T_{-x}$, which in turn, since $X$ is continuous at $T_{-x}$, implies that $X_{\gamma(T_{-x_n})} \downarrow X_{T_{-x}} = -x$. 
The latter is not possible since $P$-a.s., $\sup_{[0,T_{-x_n}]} X\geq 0$ for every $n$. 
Hence, we can conclude that a.s. $\ell = \gamma(T_{-x},X)$, i.e.
\begin{align*}
\gamma\left( T_{-x_n}, X\right) \downarrow \gamma(T_{-x},X).
\end{align*}
Moreover, since for $T_{-x}$ we also have that $P$-a.s., $\sup_{[0,T_{-x}]} X\geq 0$, we can ensure that $\gamma(T_{-x},X)<T_{-x}$, so the sequence $\left(\gamma\left( T_{-x_n}, X\right)\right)$ is not only convergent, but it is constant from some $N\geq 0$.
As a consequence, we have that a.s. for all $n\geq N$ and all $M> 0$ 
\begin{align*}
{\theta'_{\gamma(V)}\circ k_{T_{-x_n}}\circ X = \theta'_{\gamma(T_{-x})}\circ k_{T_{-x_n}}\circ X},
\end{align*}
\begin{align*}
\Vert \theta'_{\gamma(V)}\circ k_{T_{-x_n}}\circ X\circ \lambda_n - \theta'_{\gamma(V)}\circ k_{T_{-x}}\circ X \Vert_M
\leq \Vert k_{T_{-x_n}}\circ X\circ \lambda_n - k_{T_{-x}}\circ X \Vert_M \rightarrow 0.
\end{align*}

These arguments, together with the continuous mapping theorem applied to $h\in\mathcal{C}_b(\mathcal{E},\R_+)$, lead to the a.s. convergence of $h\left( \theta'_{\gamma(V)}\circ k_{T_{-x_n}}\circ X \right)$ to $h\left( \theta'_{\gamma(V)}\circ k_{T_{-x}}\circ X \right)$. 
Finally, since $h$ is bounded, the dominated convergence theorem applies, and we can conclude that
\begin{align*}
\lim_n z(x_n) = z(x),
\end{align*}
that is, $z$ is right-continuous at $x>0$. Since $x$ is arbitrary, the result is proved.
\end{proof}

\begin{proof}[Proof of Lemma~\ref{lemma:continuity-w}]
As in the proof of Lemma~\ref{lemma:continuity-z}, we can apply \cite[Proposition VI.2.11 and VI.2.12]{JaShi03}, which imply that $T_{-x}(\varepsilon)$ and even more, $k_{T_{-x}}\circ \varepsilon$, are continuous functions for every $x>0$, where $\varepsilon$ is the canonical excursion of $X-S$ away from 0.
Since $F\in\mathcal{C}_b(\mathcal{E},\R_+)$, the function $F(e)\1_{V(e)<+\infty}$ is also continuous for $e\in \mathcal{E}$.
Besides, we have for every $x>0$, that $0<\overline{n}(T_{-x}<+\infty)<+\infty$, which allows us to conclude.
\end{proof}

\section*{Acknowledgments}
This work was supported by grants from Région Ile-de-France and Labex \mbox{MemoLife} from École Normale Supérieure. The authors also thank the Center for Interdisciplinary Research in Biology (CIRB, Collège de France) for funding.


\bibliographystyle{alpha} 
\bibliography{biblio-clean} 

\begin{thebibliography}{CLUB09}

\bibitem[AD09]{AbDe09}
Romain Abraham and Jean-François Delmas.
\newblock Williams{'} decomposition of the {L}évy continuum random tree and
  simultaneous extinction probability for populations with neutral mutations.
\newblock {\em Stochastic Processes and their Applications}, 119(4):1124 --
  1143, 2009.

\bibitem[Ald91]{Ald91}
David Aldous.
\newblock The continuum random tree. {I}.
\newblock {\em Ann. Probab.}, 19(1):1--28, 1991.

\bibitem[Ald93]{Ald93}
David Aldous.
\newblock The continuum random tree. {III}.
\newblock {\em Ann. Probab.}, 21(1):248--289, 1993.

\bibitem[AN72]{AthNe72}
Krishna~B. Athreya and Peter~E. Ney.
\newblock {\em Branching processes}.
\newblock Springer-Verlag, New York-Heidelberg, 1972.
\newblock Die Grundlehren der mathematischen Wissenschaften, Band 196.

\bibitem[AP05]{AlPo05}
David Aldous and Lea Popovic.
\newblock A critical branching process model for biodiversity.
\newblock {\em Adv. in Appl. Probab.}, 37(4):1094--1115, 2005.

\bibitem[AR02]{AlRo02}
Gerold Alsmeyer and Uwe R{\"o}sler.
\newblock Asexual versus promiscuous bisexual {G}alton-{W}atson processes: the
  extinction probability ratio.
\newblock {\em Ann. Appl. Probab.}, 12(1):125--142, 2002.

\bibitem[BD16]{BiDe16}
Hongwei Bi and Jean-Fran{\c{c}}ois Delmas.
\newblock Total length of the genealogical tree for quadratic stationary
  continuous-state branching processes.
\newblock {\em Ann. Inst. Henri Poincar\'e Probab. Stat.}, 52(3):1321--1350,
  2016.

\bibitem[Ber93]{Ber93}
Jean Bertoin.
\newblock Splitting at the infimum and excursions in half-lines for random
  walks and {L}\'evy processes.
\newblock {\em Stochastic Process. Appl.}, 47(1):17--35, 1993.

\bibitem[Ber96]{Ber96}
Jean Bertoin.
\newblock {\em L\'evy processes}, volume 121 of {\em Cambridge Tracts in
  Mathematics}.
\newblock Cambridge University Press, Cambridge, 1996.

\bibitem[CD05]{ChDo05}
Loïc Chaumont and Ronald~A. Doney.
\newblock On {L}\'evy processes conditioned to stay positive.
\newblock {\em Electron. J. Probab.}, 10:no. 28, 948--961, 2005.

\bibitem[Cha94]{Cha94}
Loïc Chaumont.
\newblock Sur certains processus de {L}\'evy conditionn\'es \`a rester
  positifs.
\newblock {\em Stochastics Stochastics Rep.}, 47(1-2):1--20, 1994.

\bibitem[Cha96]{Cha96}
Loïc Chaumont.
\newblock Conditionings and path decompositions for {L}\'evy processes.
\newblock {\em Stochastic Process. Appl.}, 64(1):39--54, 1996.

\bibitem[Cha13]{Cha13}
Loïc Chaumont.
\newblock On the law of the supremum of {L}\'evy processes.
\newblock {\em Ann. Probab.}, 41(3A):1191--1217, 2013.

\bibitem[CLUB09]{CaLaUB09}
Ma.~Emilia Caballero, Amaury Lambert, and Ger{\'{o}}nimo Uribe~Bravo.
\newblock Proof(s) of the {L}amperti representation of continuous-state
  branching processes.
\newblock {\em Probab. Surveys}, 6(0):62--89, 2009.

\bibitem[DFL15]{DaLa15}
Miraine D{\'a}vila~Felipe and Amaury Lambert.
\newblock Time reversal dualities for some random forests.
\newblock {\em ALEA Lat. Am. J. Probab. Math. Stat.}, 12(1):399--426, 2015.

\bibitem[DH13]{HeDe13}
Jean-Fran{\c{c}}ois Delmas and Olivier H{\'{e}}nard.
\newblock A {W}illiams decomposition for spatially dependent superprocesses.
\newblock {\em Electron. J. Probab.}, 18(0), 2013.

\bibitem[DLG02]{DuGa02}
Thomas Duquesne and Jean-Fran{\c{c}}ois Le~Gall.
\newblock Random trees, {L}\'evy processes and spatial branching processes.
\newblock {\em Ast\'erisque}, (281):vi+147, 2002.

\bibitem[Don07]{Don07}
Ronald~A. Doney.
\newblock {\em Fluctuation theory for {L}\'evy processes}, volume 1897 of {\em
  Lecture Notes in Mathematics}.
\newblock Springer, Berlin, 2007.

\bibitem[Duq03]{Duq03}
Thomas Duquesne.
\newblock Path decompositions for real {L}evy processes.
\newblock {\em Ann. Inst. H. Poincar\'e Probab. Statist.}, 39(2):339--370,
  2003.

\bibitem[Est75]{Est75}
Warren~W. Esty.
\newblock The reverse {G}alton-{W}atson process.
\newblock {\em J. Appl. Probability}, 12(3):574--580, 1975.

\bibitem[Eva07]{Eva07}
Steven~N. Evans.
\newblock {\em Probability and Real Trees: {\'E}cole d'{\'E}t{\'e} de
  Probabilit{\'e}s de Saint-Flour XXXV-2005}.
\newblock Lecture Notes in Mathematics. Springer Berlin Heidelberg, 2007.

\bibitem[GP80]{GrPi80}
Priscilla Greenwood and Jim Pitman.
\newblock Fluctuation identities for {L}\'evy processes and splitting at the
  maximum.
\newblock {\em Adv. in Appl. Probab.}, 12(4):893--902, 1980.

\bibitem[Jag75]{Jag75}
Peter Jagers.
\newblock {\em Branching processes with biological applications}.
\newblock Wiley-Interscience [John Wiley \& Sons], London-New York-Sydney,
  1975.
\newblock Wiley Series in Probability and Mathematical Statistics---Applied
  Probability and Statistics.

\bibitem[JS03]{JaShi03}
Jean Jacod and Albert~N. Shiryaev.
\newblock {\em Limit theorems for stochastic processes}, volume 288 of {\em
  Grundlehren der Mathematischen Wissenschaften [Fundamental Principles of
  Mathematical Sciences]}.
\newblock Springer-Verlag, Berlin, second edition, 2003.

\bibitem[KRS07]{KlRoSa07}
Fima~C. Klebaner, Uwe Rösler, and Serik Sagitov.
\newblock Transformations of galton-watson processes and linear fractional
  reproduction.
\newblock {\em Advances in Applied Probability}, 39(4):1036--1053, 2007.

\bibitem[Kyp06]{Kyp06}
Andreas~E. Kyprianou.
\newblock {\em Introductory lectures on fluctuations of {L}\'evy processes with
  applications}.
\newblock Universitext. Springer-Verlag, Berlin, 2006.

\bibitem[Lam67]{Lam67}
John Lamperti.
\newblock Continuous state branching processes.
\newblock {\em Bull. Amer. Math. Soc.}, 73:382--386, 1967.

\bibitem[Lam10]{Lam10}
Amaury Lambert.
\newblock The contour of splitting trees is a {L}\'evy process.
\newblock {\em Ann. Probab.}, 38(1):348--395, 2010.

\bibitem[LG05]{Gall05}
Jean-Fran{\c{c}}ois Le~Gall.
\newblock Random trees and applications.
\newblock {\em Probab. Surv.}, 2:245--311, 2005.

\bibitem[LGLJ98]{GaJa98}
Jean-Francois Le~Gall and Yves Le~Jan.
\newblock Branching processes in {L}\'evy processes: the exploration process.
\newblock {\em Ann. Probab.}, 26(1):213--252, 1998.

\bibitem[LSZ13]{LaSiZw13}
Amaury Lambert, Florian Simatos, and Bert Zwart.
\newblock Scaling limits via excursion theory: interplay between
  {C}rump-{M}ode-{J}agers branching processes and processor-sharing queues.
\newblock {\em Ann. Appl. Probab.}, 23(6):2357--2381, 2013.

\bibitem[LUB16]{LaUB16}
Amaury Lambert and Gerónimo Uribe~Bravo.
\newblock Totally ordered measured trees and splitting trees with infinite
  variation, 2016.

\bibitem[Mie01]{Mie01}
Gr{\'e}gory Miermont.
\newblock Ordered additive coalescent and fragmentations associated to {L}evy
  processes with no positive jumps.
\newblock {\em Electron. J. Probab.}, 6:no.\ 14, 33 pp. (electronic), 2001.

\bibitem[Mil73]{Mil73}
Pressley~Warwick Millar.
\newblock Exit properties of stochastic processes with stationary independent
  increments.
\newblock {\em Trans. Amer. Math. Soc.}, 178:459--479, 1973.

\bibitem[Mil77a]{Mil77b}
Pressley~Warwick Millar.
\newblock Random times and decomposition theorems.
\newblock In {\em Probability ({P}roc. {S}ympos. {P}ure {M}ath., {V}ol. {XXXI},
  {U}niv. {I}llinois, {U}rbana, {I}ll., 1976)}, pages 91--103. Amer. Math.
  Soc., Providence, R. I., 1977.

\bibitem[Mil77b]{Mil77a}
Pressley~Warwick Millar.
\newblock Zero-one laws and the minimum of a {M}arkov process.
\newblock {\em Trans. Amer. Math. Soc.}, 226:365--391, 1977.

\bibitem[PW11]{PaWa11}
Etienne Pardoux and Anton Wakolbinger.
\newblock From {B}rownian motion with a local time drift to {F}eller's
  branching diffusion with logistic growth.
\newblock {\em Electronic Communications in Probability}, 16(0):720--731, 2011.

\bibitem[Rog66]{Rog66}
Boris~Alekseevich Rogozin.
\newblock On distributions of functionals related to boundary problems for
  processes with independent increments.
\newblock {\em Theory of Probability {\&} Its Applications}, 11(4):580--591,
  jan 1966.

\bibitem[RY91]{ReYo91}
Daniel Revuz and Marc Yor.
\newblock {\em Continuous martingales and {B}rownian motion}, volume 293 of
  {\em Grundlehren der Mathematischen Wissenschaften [Fundamental Principles of
  Mathematical Sciences]}.
\newblock Springer-Verlag, Berlin, 1991.

\bibitem[Wil74]{Will74}
David Williams.
\newblock Path decomposition and continuity of local time for one-dimensional
  diffusions. {I}.
\newblock {\em Proc. London Math. Soc. (3)}, 28:738--768, 1974.

\bibitem[Zay96]{Zay96}
Ahmed~I. Zayed.
\newblock {\em Handbook of function and generalized function transformations}.
\newblock Mathematical Sciences Reference Series. CRC Press, Boca Raton, FL,
  1996.

\end{thebibliography}


\end{document}